\documentclass[11pt,reqno]{amsart}
\usepackage{amsthm}
\usepackage{amssymb}
\usepackage{latexsym}
\usepackage{multicol}
\usepackage{verbatim,enumerate}
\usepackage[usenames]{color}
\usepackage[all, poly]{xy}
\usepackage[utf8]{inputenc}
\usepackage{hyperref}
\usepackage{amsmath, amscd}
\usepackage{soul}
\usepackage{tikz}
\usetikzlibrary{matrix,arrows,decorations.pathmorphing}
\usetikzlibrary{arrows}
\allowdisplaybreaks
\usepackage{dirtytalk,caption,subcaption}
\usetikzlibrary{positioning}
\usepackage{float}
\usepackage{wasysym}

\theoremstyle{plain}
\newtheorem*{prop}{Proposition}
\newtheorem{thm}{Theorem}
\newtheorem*{lem}{Lemma}
\newtheorem*{cor}{Corollary}

\theoremstyle{definition}
\newtheorem*{example}{Example}
\newtheorem*{defn}{Definition}

\newtheorem*{rem}{Remark}

\theoremstyle{remark}

\newcommand{\lie}[1]{\mathfrak{#1}}

\advance\textwidth by 1.2in  \advance\oddsidemargin by -.6in \advance\evensidemargin by -.6in

\newcounter{cnt}
 \makeatletter
\def\mydggeometry{\makeatletter\dg@YGRID=1\dg@XGRID=20\unitlength=0.003pt\makeatother}
\makeatother \theoremstyle{remark}

\numberwithin{equation}{section}
\makeatletter
\def\section{\def\@secnumfont{\mdseries}\@startsection{section}{1}%
  \z@{.7\linespacing\@plus\linespacing}{.5\linespacing}%
  {\normalfont\scshape\centering}}
\def\subsection{\def\@secnumfont{\bfseries}\@startsection{subsection}{2}%
  {\parindent}{.5\linespacing\@plus.7\linespacing}{-.5em}%
  {\normalfont\bfseries}}
\makeatother

\begin{document}

\title[Marked multi-colorings, Lie superalgebras and right-angled Coxeter groups]{Marked multi-colorings, partially commutative Lie superalgebras and right-angled Coxeter groups}

\author{Chaithra P}\address{Department of Mathematics, Indian Institute of Science, Bangalore 560012, India}
\email{chaithrap@iisc.ac.in}
\thanks{C.P. acknowledges funding from the Prime minister's research fellowship (TF/PMRF-22-5467.03) for carrying out this research.}

\author{Deniz Kus}\address{Faculty of Mathematics\\ University of Bochum\\ Universit{\"a}tsstr. 150, Bochum 44801, 
Germany}
\email{deniz.kus@rub.de}
\thanks{}

\author{R. Venkatesh}
\address{Department of Mathematics, Indian Institute of Science, Bangalore 560012, India}
\address{Visiting Faculty, Department of Mathematics, Indian Institute of Technology Madras, Chennai 600036, India}
\email{rvenkat@iisc.ac.in}
\thanks{R.V. was partially supported by the ANRF Core Grant:
ANRF/CRG/2023/008780.}

\subjclass[2020]{}

\begin{abstract}
Infinite-dimensional Lie superalgebras, particularly Borcherds-Kac-Moody (BKM) superalgebras, play a fundamental role in mathematical physics, number theory, and representation theory. In this paper, we study the root multiplicities of BKM superalgebras via their denominator identities, deriving explicit combinatorial formulas in terms of graph invariants associated with marked (quasi) Dynkin diagrams. We introduce partially commutative Lie superalgebras (PCLSAs) and provide a direct combinatorial proof of their denominator identity, where the generating set runs over the super heaps monoid. A key notation in our approach is marked multi-colorings and their associated polynomials, which generalize chromatic polynomials and offer a method for computing root multiplicities. As applications, we characterize the roots of PCLSAs and establish connections between their universal enveloping algebras and right-angled Coxeter groups, leading to explicit formulas for their Hilbert series. These results further deepen the interplay between Lie superalgebras, graph theory, and algebraic combinatorics.

\end{abstract}

\maketitle

\section{Introduction}
The study of infinite-dimensional Lie (super)algebras has long been an extensive area of research, particularly in the context of Borcherds-Kac-Moody (BKM) superalgebras, which are natural extensions of BKM Lie algebras and incorporate a $\mathbb Z_2$-grading, distinguishing between even and odd elements. These infinite-dimensional Lie superalgebras generalize both Kac-Moody superalgebras and Borcherds algebras by allowing imaginary simple roots and additional algebraic structures that arise in the superalgebraic setting. Their rich structure has led to numerous applications in mathematical physics, number theory, representation theory, and beyond (see, for example, \cite{Borcherds1992, Ray06, Scheithauerbook, Wa01} and references therein). 

They also play a key role in the study of vertex superalgebras, which are used to model the symmetries of superstring theory and supergravity \cite{Huang2020, Moriwaki2023, Xu98}. These algebras provide a natural framework for understanding extended supersymmetry in quantum field theory and integrable systems.
BKM superalgebras also exhibit deep connections with automorphic forms and supermodular invariants through their denominator identities \cite{Ray06}. These identities encode crucial structural information about the algebra, such as root multiplicities. 


The main motivation of this paper is to understand the root multiplicities encoded in the denominator identity of BKM  superalgebras and to provide explicit formulas for these multiplicities in terms of graph invariants associated with the underlying (quasi) Dynkin diagrams. These connections build on prior work linking the chromatic polynomial of a graph to the root multiplicities of the associated Kac-Moody and Borcherds algebras (see \cite{AKV18}, \cite{SA2024} and \cite{VV2015}).

The main ingredient in our study are partially commutative Lie superalgebras (PCLSA), which are related to BKM superalgebras in the following ways, for example:
\begin{itemize}
    \item Partially commutative Lie superalgebras can be realized as the positive part of BKM superalgebras when all simple roots are considered to be imaginary. Recall that the positive part of a BKM superalgebra is the subalgebra generated by the root spaces corresponding to positive roots.
\item The multiplicities of free roots of BKM superalgebras are related to the dimensions of the graded spaces of partially commutative Lie superalgebras, as explained in \cite{SA2024}.
\item The free Lie superalgebras naturally appear in the decomposition of BKM  superalgebras that are obtained from Lazard's elimination theorem (see \cite{Jurisich2004} and \cite[Proposition 3.6]{Kang1998}).  
\end{itemize}
Let $\mathcal{G}$ be a simple graph with a countable vertex set $I$ and a countable edge set $E$. Motivated by the theory of Lie superalgebras, we treat the graph $\mathcal{G}$ as a marked graph as follows: given $I^{\mathrm{it}}_1 \subseteq I_1 \subseteq I$, we define the elements of $I_1$, $I_0 := I \setminus I_1$, and $I^{\mathrm{it}}_1$ as the odd vertices, even vertices, and odd isotropic vertices of $\mathcal{G}$, respectively.

The partially commutative Lie superalgebra $\mathcal{P}(\mathcal{G})$ associated with a marked graph $\mathcal{G}$ and vertex set $I$ is defined as the quotient $\mathcal{P}(\mathcal{G}) = \mathcal{F}(X_I) / \mathcal{J}$, where $X_I$ is a $\mathbb{Z}_2$-graded set with
$(X_I)_1=\{e_i: i\in I_1\}$, $(X_I)_0=\{e_i: i\in I_0\}$ and $\mathcal{J}$ is the ideal in $\mathcal{F}(X_I)$ generated by 
    $$\left\{[e_i,e_j] : i\neq j,\ (i,j)\notin E\right\}\cup\left\{[e_i,e_i]: i\in I_1^{\mathrm{it}}\right\}.$$
Here, we denote by $\mathcal{F}(X_I)$ the free Lie superalgebra generated by $X_I$; see Section~\ref{defpcla} for more details. The last relation was not included in \cite{SA2024} in their definition of PCLSA, as it is not necessary for the study of free roots of BKM superalgebras. However, it is natural to include these relations, since they hold in BKM superalgebras. Moreover, including them will also enable us to consider a much larger class of roots of BKM superalgebras. More specifically, this will allow us to provide explicit formulas for the root multiplicities of all roots, which only have restrictions on the coefficients of real simple roots. For further details, see Section~\ref{BKMconnections}.

In a first step we provide a more direct and accessible combinatorial proof of the denominator identity of PCLAs, using ideas from heap theory introduced by Viennot (see \cite{Vi89}). This allows us to connect the marked multivariate independence series of the graph $\mathcal{G}$ with root multiplicities of PCLAs.  Indeed, we have (see Proposition~\ref{denomi})
    $$\mathcal{I}(\mathcal{G},-\mathbf{x})^{-1}=\frac{\prod_{\beta\in \Delta_1^{\mathcal{G}}}(1+x^\beta)^{\mathrm{mult}_{\mathcal{G}}(\beta)}}{\prod_{\beta\in \Delta_0^{\mathcal{G}}}(1-x^\beta)^{\mathrm{mult}_{\mathcal{G}}(\beta)}}$$
See Section~\ref{denominatoridentity} for the notations used in the above identity. One of the main notions of this paper is the introduction of marked multi-colorings and their associated polynomials, which generalize classical chromatic polynomials and offer a new tool for computing root multiplicities. For more details, see Section \ref{markedcolorings}. We have the following identity in terms of marked chromatic polynomials for any $q\in\mathbb{Z}$ (see Theorem~\ref{expmarkedindp}):
    $$\mathcal{I}(\mathcal{G},\mathbf{x})^q=\sum_{\mathbf{m}\in \mathbb{Z}_+^{I}} {_{\bold{m}}\Pi^{\mathrm{mark}}_\mathcal{G}(q)}\ x^\mathbf{m}.$$ 
In Section~\ref{markedcolorings}, we compute the marked chromatic polynomials for several families of marked graphs, such as star graphs and so-called PEO graphs (graphs with a perfect elimination ordering on their vertices; see Theorem \ref{peomarked}). The following result explicitly expresses the root multiplicities of PCLAs (and hence large classes of root multiplicites for BKM superalgebras as described above) in terms of marked chromatic polynomials (see Theorem~\ref{multroo1}):
$$ \mathrm{mult}_{\mathcal{G}}(\bold m) = \epsilon\big(\bold m)\sum\limits_{\ell | \bold m}\frac{\mu(\ell)}{\ell}\epsilon\big(\bold m/\ell\big)(-1)^{|\bold m|/\ell-1} {_{\frac{\bold{m}}{\ell}}\Pi^{\mathrm{mark}}_\mathcal{G}}(q)[q].$$
Thus, together with the explicit expression of marked chromatic polynomials (for example, for PEO graphs), this provides a closed formula for the root multiplicities. 

In conclusion, we present two applications of our results. The first (Proposition~\ref{rootdescp}) provides an explicit characterization of the roots of PCLAs. Although PCLSAs are closely related with BKM superalgebras, we cannot use the description of roots from Wakimoto's book \cite[Theorem 2.46, Page 120]{Wa01}, as there appear to be missing assumptions in the particular case of PCLAs. In contrast to partially commutative Lie algebras or the PCLAs described in \cite{SA2024}, the description of roots is more involved and does not depend solely on whether the support of an element in the root lattice is connected.

The second application (Theorem~\ref{racghilber}) reveals a connection between the universal enveloping algebras of PCLAs and right-angled Coxeter groups. This connection enables us to compute the Hilbert series of these groups explicitly in terms of the marked independence series. While the one-variable case of this result (i.e., the Poincaré series) is well-established in the literature \cite{Davis2015}, the multivariate generalization appears to be a new contribution. More precisely, for the Hilbert series $\mathcal{H}_{\mathcal{G}}(\bold x)$ of the right-angled Coxeter group $W_{\mathcal{G}}$, we have:
$$\mathcal{H}_{\mathcal{G}}(\bold x) = \sum\limits_{w\in W_{\mathcal{G}}} \bold x(w) = \frac{1}{I({\mathcal{G}},-\bold x)}$$
where $\bold x(w) = x_{i_1}\cdots x_{i_r}$ if $w = s_{i_1}\cdots s_{i_r}$ is a reduced expression of $w\in W_{\mathcal{G}}$.

\medskip
The paper is organized as follows: Section \ref{denominatoridentity} introduces the necessary notation, defines partially commutative Lie superalgebras and marked independence series, and establishes the denominator identities for these algebras. In Section~\ref{markedcolorings}, we introduce marked colorings and marked chromatic polynomials, exploring their connection to marked independence series. Additionally, we compute the marked chromatic polynomials for PEO-graphs. Section~\ref{rootmultisection} presents explicit formulas expressing root multiplicities in terms of marked chromatic polynomials and vice versa. In Section~\ref{rootssection}, we give explicit descriptions of the roots of PCLSAs.
Finally, in Section~\ref{racgsection}, we apply our results to compute the Hilbert series of right-angled Coxeter groups.

\medskip

\textit{Acknowledgement:}  D.K. thanks Minoru Wakimoto for valuable discussions on the root spaces of BKM superalgebras.  

\section{Lie superalgebras: Inversion lemma and denominator identity}\label{denominatoridentity}

\subsection{}
We denote by $\mathbb{N}, \mathbb{Z}_{+}$, and $\mathbb{Z}$ the set of positive integers, non-negative integers, and integers, respectively. Unless otherwise stated, all vector spaces are assumed to be complex vector spaces. For a set $S$, we denote by $P^{\mathrm{mult}}(S)$ the collection of all finite multi-subsets of $S$ and by $\mathbb{Z}^S$ the set of all finitely-supported $S$-tuples of integers where the support of a tuple $\bold{m}=(m_i:i\in S)$ is defined by $\mathrm{supp}(\bold{m})=\{i\in S : m_i\neq 0\}$. The corresponding analogues $\mathbb{Z}^S_+$ and $\mathbb{N}^S_{}$ respectively are defined similarly. For a multi-subset $U$, we denote by $|U|$ the cardinality of $U$; the elements are counted with multiplicity and given $\bold{m}=(m_i:i\in S)\in\mathbb{Z}^S_{+},$ we set $|\bold{m}|=\sum_{i\in S}m_i$.
\subsection{} Let $\mathbb{Z}_2 = \{0, 1\}$ by the group of integers modulo $2.$
In this section, we recall the basic theory of Lie superalgebra and define free partially commutative Lie superalgebras.
\begin{defn}
  A Lie superalgebra $\mathfrak{g}$ is a $\mathbb{Z}_2$-graded vector space $\mathfrak{g}=\mathfrak{g}_0\oplus\mathfrak{g}_1$ equipped with a Lie product operation $[-,-]:\mathfrak{g}\times\mathfrak{g}\rightarrow \mathfrak{g}$ satisfying the following conditions: for all homogeneous elements $x,y$ and $z$ we have
  \begin{enumerate}
      \item $[x,y]=-(-1)^{|x||y|}[y,x]$
      \item $[x,[y,z]]=[[x,y],z]+(-1)^{|x||y|}[y,[x,z]]$
  \end{enumerate}
  where $|\cdot|$ denotes the parity.
\end{defn}
The elements in $\mathfrak{g}_0$ (resp. $\mathfrak{g}_1$) are called even (resp. odd). From the definition, it follows that $\mathfrak{g}_0$ is a Lie algebra and $\mathfrak{g}_1$ is a module over $\mathfrak{g}_0$. Any associative superalgebra $A$ is naturally a Lie superalgebra with bracket 
\begin{equation}\label{asslie}[a,b]=ab-(-1)^{|a||b|}ba,\ \ a,b\in A\ \text{ homogeneous}.\end{equation}
\subsection{}
We recall the Poincaré-Birkoff-Witt theorem for Lie superalgebras. Let $\mathbf{T}(\mathfrak{g})$ the tensor superalgebra and let $\mathbf{J}$ the ideal of $\mathbf{T}(\mathfrak{g})$ generated by the elements of the form 
$$[x,y]-x\otimes y+(-1)^{|x||y|}y\otimes x,\ \ x,y\in\lie g,\text{ homogeneous}$$
The universal enveloping algebra $\mathbf{U}(\mathfrak{g})$ is an associative superalgebra defined as the quotient $\mathbf{T}(\mathfrak{g})/\mathbf{J}$. The PBW basis is obtained as follows (see for example \cite[Theorem 6.1.1]{Mu12}). Let $Z_0$ be a basis of $\mathfrak{g}_0$ and $Z_1$ be a basis of $\mathfrak{g}_1$ and fix a total order $\leq $ on $Z=Z_0\cup Z_1$. Then the set of monomials 
$$z_1z_2\cdots z_n,\ \ z_i\in Z,\ z_i\leq z_{i+1},\ \text{and }z_i<z_{i+1} \text{ if $z_i\in Z_1$}$$
forms a vector space basis of $\mathbf{U}(\mathfrak{g})$.

The main Lie superalgebra considered in this paper will be the partially commutative Lie superalgebra which we introduce next. We will give in this setting a direct combinatorial proof of the denominator identity for Borcherds-Kac-Moody Lie superalgebras \cite[Section 2.5]{Wa01} and apply this in the following sections to derive formulas for their root multiplicities in terms of marked chromatic polynomials and describe the set of roots explicitly.

\subsection{}\label{defpcla} Let $\mathcal{G}$ be a simple graph with countable vertex set $I$ and countable edge set $E$.  Recall that a graph with no multiple edges and loops is called simple. 
For any two vertices $i, j\in I$, we denote by $e(i, j)\in E$ the edge between $i$ and $j$ (if an edge exists). Motivated by the theory of Lie superalgebras, we treat the graph $\mathcal{G}$ as a marked graph as follows: given $I^{\mathrm{it}}_1\subseteq I_1\subseteq I$ we call the elements of $I_1$, $I_0:=I\backslash I_1$ and $I^{\mathrm{it}}_1$ the odd vertices, even vertices, and odd isotropic vertices of $\mathcal{G}$, respectively. In the following we will present the odd isotropic vertices by $\Circle$, and the other vertices are represented by $\CIRCLE$. 

For a marked graph $(\mathcal{G},I_1,I^{\mathrm{it}}_1)$, we can define the associated free partially commutative Lie superalgebra $\mathcal{P}(\mathcal{G})$ (we drop in the rest of the paper the dependence on the marking for simplicity) as follows. First, we will introduce the free Lie superalgebra on a $\mathbb{Z}_2$ graded set $X$.
\begin{defn}
    Let $X=X_0\cup X_1$ be a non-empty $\mathbb{Z}_2$ graded set and $X^*$ be the free monoid generated by $X$. A word $w\in X^*$ is called even (resp. odd) if the number of elements of $X_1$ appearing in $w$ is even (resp. odd). This defines the free associative superalgebra $F(X)$ on $X$, which becomes a Lie superalgebra with bracket as in \eqref{asslie}. The free Lie superalgebra $\mathcal{F}(X)$ of $X$ is defined as the Lie sub-superalgebra of $F(X)$ generated by $X$.
    \end{defn}
There are well-known bases for free Lie algebras, such as the Lyndon and Shirshov bases, which were constructed in \cite{LY54} and \cite{SH62}, respectively. A description of bases for free Lie superalgebras can be found in \cite{BMP92}, and more recently, bases have been constructed in \cite{ML24} using basic commutators.
The free partially commutative Lie superalgebra $\mathcal{P}(\mathcal{G})$ associated to a marked graph $\mathcal{G}$ with vertex set $I$ is defined as the quotient $\mathcal{P}(\mathcal{G})=\mathcal{F}(X_I)/\mathcal{J}$, where $(X_I)_1=\{e_i: i\in I_1\}$, $(X_I)_0=\{e_i: i\in I_0\}$ and $\mathcal{J}$ is the ideal in $\mathcal{F}(X_I)$ generated by 
    $$\left\{[e_i,e_j] : i\neq j,\ (i,j)\notin E\right\}\cup\left\{[e_i,e_i]: i\in I_1^{\mathrm{it}}\right\}.$$
We will later see that $\mathcal{P}(\mathcal{G})$ can be realized as the positive part of a Borcherds-Kac-Moody superalgebra.
\begin{rem}
The definition of free partially commutative Lie superalgebras defined in \cite{SA2024} is different from ours. The reason is that they are interested only in free roots of BKM superalgebras; these correspond to roots independent from the Serre relations and the coefficients of real simple roots and odd isotropic simple roots are at most one (see \cite[Definition 2.6]{SA2024}). In the situation of arbitrary root spaces one has to include the relation $[e_i,e_i]=0$ for $i\in I_1^{\mathrm{it}}.$ 
\end{rem}
\subsection{} In this subsection we define the marked multivariate independence series of marked graphs. Let $R=\mathbb{C}[[x_i : i\in I]]$ be the ring of formal power series generated by the countable set of commuting variables $\{x_i : i\in I\}$ over $\mathbb{C}$. For $\mathbf{m} = (m_i : i\in I)\in \mathbb{Z}_{+}^{I}$ and $f\in R$ we denote by $f[x^{\mathbf{m}}]$ the coefficient in front of $x^{\mathbf{m}}:=\prod_{i\in I} x_i^{m_i}$.

\begin{defn} We introduce the following notations.
\begin{enumerate}
    \item An element $U\in P^{\mathrm{mult}}(I)$ is called independent if the subgraph spanned by the underlying set of $U$ has no edges, i.e.,
    any two distinct vertices from $U$ have no edge between them.
\item We denote by $\mathcal{I}({\mathcal{G}})$ the set of all finite independent multi-subsets $U$ of $\mathcal{G}$, such that the vertices from $I\backslash I_1^{\mathrm{it}}$ appear at most once in $U$. We understand $\emptyset\in \mathcal{I}(\mathcal{G})$.
\item Given $U\in I(\mathcal{G})$, set $x(U) = \prod_{i\in I}x_i^{m_i}$, where $i\in I$ appears $m_i$ times in $U$.  
The marked multivariate independence series of the graph $\mathcal{G}$ is defined as
$$\mathcal{I}(\mathcal{G},\mathbf{x})=\sum_{U\in \mathcal{I}(\mathcal{G})} x(U).$$
\end{enumerate}
\end{defn}
Note that $I(\mathcal{G},\mathbf{x})^q\in R$ for all $q\in \mathbb Z$. 
\begin{example}\begin{enumerate}
    \item We consider the path graph $\mathcal{G}$ with $I_1^{\mathrm{it}}=\{4\}$: 
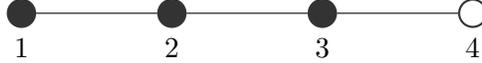
    \begin{figure}[h]
    \centering
    \begin{tikzpicture}
    \tikzstyle{B}=[circle,draw=black!80,fill=black!80,thick]
        \tikzstyle{C}=[circle,draw=black!80,thick]
    \node[B] (1) at (0,0) [label=below:$1$]{};
    \node[B] (2) at (2,0) [label=below:$2$]{};
    \node[B] (3) at (4,0) [label=below:$3$]{};
    \node[C] (4) at (6,0) [label=below:$4$]{};
    \path[-] (1) edge node[left]{} (2);
    \path[-] (2) edge node[left]{} (3);
    \path[-] (3) edge node[left]{} (4);
\end{tikzpicture}
    \label{}
\end{figure}

The marked multivariate independence series of this graph is $$\mathcal{I}(\mathcal{G},\mathbf{x})=1+x_1+x_2+x_3+\sum_{m=1}^\infty x^m_4+x_1x_3+\sum_{m=1}^\infty x_2 x^m_4+\sum_{m=1}^\infty x_1x^m_4.$$
\item The marked multivariate independence series of the marked graph $\mathcal{G}$ 
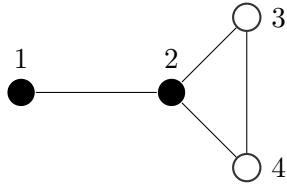
\begin{figure}[h]
\centering
\begin{subfigure}{0.49\textwidth}
\centering
    \begin{tikzpicture}
        \node (1) at (1,4) [circle, fill=black, thick, label=above:1] {};
    \node (2) at (3,4) [circle, fill=black,thick, label=above:2] {};
    \node (3) at (4,5) [circle, draw=black!80, thick, label=right:3] {};
    \node (4) at (4,3) [circle,draw=black!80, thick, label=right:4] {};

    \path[-] (1) edge node[left]{} (2);
    \path[-] (2) edge node[left]{} (3);
    \path[-] (2) edge node[left]{} (4);
    \path[-] (3) edge node[left]{} (4);
\end{tikzpicture}
\end{subfigure}
\end{figure}

is given by
$$\mathcal{I}(\mathcal{G},\mathbf{x})=1+x_1+x_2+\sum_{m=1}^{\infty}x_3^m+\sum_{m=1}^{\infty}x_4^m+\sum_{m=1}^{\infty}x_1x_3^m+\sum_{m=1}^{\infty}x_1x_4^m.$$
    \end{enumerate} 
\end{example}
\subsection{} Fix a total order on $I$ and consider the associated lexicographical order on the free monoid $I^*$. We consider the Cartier-Foata monoid
$\mathcal{M}(\mathcal{G})=I^*/\sim$ where $\sim$ is generated by the relations $ab\sim ba,\ (a,b)\notin E$. Note that $\mathcal{M}(\mathcal{G})$ has a natural $\mathbb{Z}_2$-gradation induced from the $\mathbb{Z}_2$-gradation of $I^*$. Given $[a]\in \mathcal{M}(\mathcal{G})$, let $a_{\mathrm{max}}\in I^{*}$ the unique maximal element in the equivalence class of $a$ with respect to the lexicographical order. We have a total order on $\mathcal{M}(\mathcal{G})$ given by $$[a]\leq [b]:\iff a_{\mathrm{max}}\leq b_{\mathrm{max}}.$$
Furthermore, given $w=i_1i_2\cdots i_r\in I^*$ we set as usual $|w|=r$, $i(w)=|\{j:i_j=i\}|$ for all $i\in I$ and $\mathrm{supp}(w)=\{i\in I: i(w)\neq 0\}$.  
 \begin{defn} Let $[w]\in \mathcal{M}(\mathcal{G})$. The initial multiplicity of $i\in I$ in $[w]$, denoted by $a_i(w)$, is defined to be the largest $k\geq 0$ for which there exists $u\in I^*$ with $[w]=[i^ku]$. The initial alphabet of $[w]$ is the multiset
$$\mathrm{IA}_m(w)=\left\{i^{a_i(w)}: i\in I\right\}$$
and the underlying set is denoted by $\mathrm{IA}(w)$. Similarly, we define the ending alphabet $\mathrm{EA}_m(w)$ and $\mathrm{EA}(w)$ respectively.
\end{defn} 
We set
$$\mathcal{M}'(\mathcal{G})=\left\{[w]\in \mathcal{M}(\mathcal{G}): \text{for all }  [w]=[w_1w_2], \ \text{we have } a_i(w_2)\leq 1,\  \forall  i\in I_1^{\mathrm{it}} \right\}$$
and define a monoid structure on $\mathcal{M}'(\mathcal{G})\cup\{0\}$ by
$$[w][w']=\begin{cases}
    [ww'],&\text{ if $\mathrm{EA}(w)\cap \mathrm{IA}(w')\cap I_1^{\mathrm{it}}=\emptyset$}\\
    0,&\text{ otherwise}
\end{cases}$$
and $[w]\cdot 0=0\cdot [w]=0$. It is straightforward to check that the empty word is the identity element and the product is associative.


For $w=i_1i_2\dots i_r\in I^*$  let $e_{w}=e_{i_1}e_{i_2}\cdots e_{i_r}$. Then we have the following basis. 
\begin{prop}\label{monoid basis}
Let $\mathcal{G}$ be a marked graph. Then 
\begin{equation}\label{ba23}\{e_{w_{\mathrm{max}}}: [w]\in \mathcal{M}'(\mathcal{G})\}\end{equation}
forms a basis for the universal enveloping algebra $\mathbf{U}_{\mathcal{G}}$ of $\mathcal{P}(\mathcal{G})$.
\end{prop}
\begin{proof} Clearly, the universal enveloping algebra is spanned by all monomials $e_w$ with $w\in I^*$ and if $[w]=[u]$, for some $w,u\in I^*$, then $e_w=\pm e_u$. Thus it is enough to choose one element of each equivalence class and we make the choice $\{e_{w_{\mathrm{max}}}: [w]\in \mathcal{M}(\mathcal{G})\}$. Further, if $[w]\notin \mathcal{M}'(\mathcal{G})$, then $e_{w_{\mathrm{max}}}=0$ by the relation $[e_i,e_i]=0$ in $\mathcal{P}(\mathcal{G})$ for all $i\in I_1^{\mathrm{it}}.$ So \eqref{ba23} gives a spanning set. 
Now consider the $\mathbb{C}$-algebra $$\hat{\mathbb{C}}=\mathbb{C}\langle z_i: i\in I\rangle/\mathcal{Z}$$
where $\mathcal{Z}$ is the two-sided ideal generated by $z_i^2, i\in I_1^{\mathrm{it}}$ and $z_iz_j-(-1)^{|i||j|}z_jz_i,\ i\neq j,\ (i,j)\notin E$. Then, $e_i\mapsto z_i$ extends to a linear map $\mathcal{P}(\mathcal{G})\rightarrow \hat{\mathbb{C}}$ and hence by the universal property an algebra homomorphism $\mathbf{U}_{\mathcal{G}}\rightarrow \hat{\mathbb{C}}$. Since, the image of the above set is linearly independent we get the claim.
\end{proof}

\begin{rem} The monoid $\mathcal{M}(\mathcal{G})$ can be identified with the monoid of heaps of pieces introduced by Viennot \cite{Vi89}. Here a pre-heap $H$ is a finite subset of $I \times \{0,1,2, \dots \}$ such that, if $(i,m),(j,n) \in H$ with $i \approx j$, then $m \ne n$, where $i \approx j$ means $i=j$ or $(i,j)\in E$. Each element is called a basic piece and for $(i,m)\in H$ we call $i$ the position and $m$ the level. Each pre-heap $H$ defines a partial order by taking the transitive closure of 
$$(i,m) \le_H (j,n) \ \text{if $i \approx j$ and }  m < n.$$

A heap $H$ is a pre-heap such that: if $(i,m) \in H$ with $m > 0$ then there exists $(j,m-1) \in H$ such that $i \approx j$. The set of all heaps is denoted by $\mathcal{H}(I,\approx)$. It is known that each pre heap has a unique heap in its isomorhism class (it will be the pre heap whose sum over its levels is minimal); here an isomorphism is a position and order preserving bijection.

The set $\mathcal{H}(I,\approx)$ is a monoid with a product called the superposition of heaps and is denoted by $H\circ H'$ (see \cite{Vi89} for details). The map $[i_1\cdots i_r]\mapsto (i_1,0)\circ \cdots \circ (i_r,0)$ defines an isomorphism of monoids $\Phi:\mathcal{M}(\mathcal{G})\rightarrow \mathcal{H}(I,\approx)$. Similarly, by setting
$\mathcal{H}'(I,\approx)=\{\Phi([w]):w\in M'(\mathcal{G})\}$ and defining the modified multiplication on $\mathcal{H}'(I,\approx)\cup \{0\}$ as follows
\begin{equation}
    H\cdot H'= \begin{cases}
    H\circ H',&\ \text{if } \max(H)\cap\min(H')\cap I^\mathrm{it}_1=\emptyset\\
    0,&\   \text{if } \max(H)\cap\min(H')\cap I^\mathrm{it}_1\neq\emptyset
\end{cases},\ \ H,H'\in \mathcal{H}'(I,\approx)
 \end{equation}
(and $0\cdot H=H\cdot 0=0$) we get an isomorphism of monoids $\mathcal{M}'(\mathcal{G})\cup\{0\}\rightarrow \mathcal{H}'(I,\approx)\cup\{0\}$.
Here $\min{(H)}$ (resp. $\max{(H)}$) is the set of minimal (resp. maximal) pieces of $H$ with respect to the partial order $\le_H$. 

We call $\mathcal{H}'(I,\approx)\cup\{0\}$ the super heaps monoid; so all results of this paper could in principle be translated to the language of super heaps as in \cite{SA2024}.
\end{rem}

\subsection{}
For a subset $K$ of $I$, let $\mathcal{M}'_K(\mathcal{G})$ be the set of all elements $[w]$ in $\mathcal{M}'(\mathcal{G})$ such that $[w]$ can only end with a vertex from $K$ (the empty word is allowed). Set
$$x(w)=\prod_{j=1}^rx_{i_j}\in R,\ \ [w]=[i_1\cdots i_r]\in \mathcal{M}(\mathcal{G}).$$
The so-called inversion lemma is a fundamental result of Viennot \cite{Vi89} that provides a closed formula for the generating function of heaps, with all maximal pieces contained in some fixed subset (see also \cite[Lemma 3.2]{BJN19}). The monoid of heaps with pieces in the vertex set of some graph is equivalent to the Cartier-Foata monoid associated with that graph. Similar ideas from Viennot’s proof lead to the analogue of the inversion lemma for $\mathcal{M}'_K(\mathcal{G})$, stated as follows. 
\begin{lem}\label{invlem} Let $\mathcal{G}$ be a marked graph with vertex set $I$ and $K\subseteq I$. The generating function for the words in $\mathcal{M}'_K(\mathcal{G})$ is given by 
\begin{equation}\label{eq: inversion lemma}
            \sum_{[w]\in \mathcal{M}'_K(\mathcal{G})}x(w)=\frac{\sum_{U\in{\mathcal{I}(\mathcal{G}(I\setminus K))}}(-1)^{|U|}x(U)}{\sum_{U\in{\mathcal{I}(\mathcal{G})}}(-1)^{|U|}x(U)}
\end{equation}
  In particular, if $K=I$, the right hand side equals $\mathcal{I}(\mathcal{G},-\mathbf{x})^{-1}$.   
\end{lem}
\begin{proof}
        Note that \ref{eq: inversion lemma} is equivalent to the following identity
        \begin{equation}\label{inversion}
            \left( \sum_{[w]\in \mathcal{M}'_K(\mathcal{G})}x(w)\right)\left( \sum_{U\in{\mathcal{I}(\mathcal{G})}}(-1)^{|U|}x(U)\right)={\sum_{U\in{\mathcal{I}(\mathcal{G}(I\setminus K))}}(-1)^{|U|}x(U)}.
        \end{equation}
We first define a few sets for a given element $[w]=[i_1i_2\cdots i_r]\in \mathcal{M}'_K(\mathcal{G})$ and $U\in \mathcal{I}(\mathcal{G})$. Recall that $\mathrm{EA}(w)\subseteq K$ denotes the set of all $i\in I$ such that $[w]=[w'i]$ for some $w'\in I^*$. Moreover, let 
$$A_{[w]}(U)=\{i\in \mathrm{IA}(w): \{i\}\cup U\in \mathcal{I}(\mathcal{G})\},$$
$$L(U,[w])=A_{[w]}(U)\cup (U\cap K)\cup \{i\in U: (i,w)\in E\}\subseteq I$$
where $(i,w)\in E$ means by definition that $(i,i_s)\in E$ for some $s\in\{1,\dots,r\}$. 
If $L(U,[w])=\emptyset$, we have $U \in \mathcal{I}(\mathcal{G}(I\setminus K))$. Moreover, $w$ must be the empty word since otherwise we have $\{i_1\}\cup U \notin \mathcal{I}(\mathcal{G})$ which means either $i_1\in U$ or $i_1\notin U$ and $i_1$ is connected to a node in $U$ in the graph $\mathcal{G}$ which is impossible. However, $i_1\in U$ would give $i_1\in \mathrm{EA}(w)\subseteq K$ which is again impossible. Hence we obtain 
\begin{equation}\label{well}L(U,[w])=\emptyset \iff U \in \mathcal{I}(\mathcal{G}(I\setminus K))\ \text{and } w=1\end{equation}
and set
$$M_V=\{(U,[w])\in \mathcal{I}(G) \times \mathcal{M}'_K(\mathcal{G}): L(U,[w])\neq \emptyset\}.$$
If $L(U,[w])\neq \emptyset$, we denote by $i_{\mathrm{min}}\in L(U,[w])$ its minimal vertex 
and define a map $\sigma:M_V\rightarrow M_V$ as follows:
$$\sigma((U,w))=\begin{cases} (U\cup \{i_{\mathrm{min}}\},[w']),& \text{if $i_{\mathrm{min}}\in A_{[w]}(U)$}\\
(U\setminus \{i_{\mathrm{min}}\},[i_{\mathrm{min}}w]),& \text{if $i_{\mathrm{min}}\notin A_{[w]}(U)$}
\end{cases}$$
where $[w]=[i_{\mathrm{min}}w']$ in the first case. The map is well-defined for the following reasons. 
\begin{itemize}
    \item If $i_{\mathrm{min}}\in A_{[w]}(U)$, then $\{i_{\mathrm{min}}\}\cup U\in \mathcal{I}(\mathcal{G})$ and clearly $[w']\in \mathcal{M}'_K(\mathcal{G})$. Further, if $w'=1$ we have  $i_{\mathrm{min}}\in K$ and by \eqref{well} the tuple lies in $M_V$. \vspace{0,1cm} 
    \item In the second case,  it will be enough to check that $i_{\mathrm{min}}\notin \mathrm{IA}(w)$ if $i_{\mathrm{min}}\in I_1^{\mathrm{it}}$ to guarantee that $[i_{\mathrm{min}}w]$ lies in $\mathcal{M}'_K(\mathcal{G})$. However, if we would have $i_{\mathrm{min}}\in \mathrm{IA}(w)$, then property $\{i_{\mathrm{min}}\}\cup U\notin \mathcal{I}(\mathcal{G})$ gives $i_{\mathrm{min}}\notin U$ which is absurd.
\end{itemize}
\textit{Claim:} The map $\sigma$ is an involution.

\textit{Proof of the Claim:}
Suppose first $i_{\mathrm{min}}\in A_{[w]}(U)$. We have to show that each element of $A_{[w']}(U\cup \{i_{\mathrm{min}}\})$ is strictly greater than $i_{\mathrm{min}}$ and each element of the complement of $A_{[w']}(U\cup \{i_{\mathrm{min}}\})$ in $L(U\cup \{i_{\mathrm{min}}\},[w'])$ is greater or equal to $i_{\mathrm{min}}$. This is enough since $\mathrm{EA}(w)\subseteq K$ implies $i_{\mathrm{min}}\in L(U\cup \{i_{\mathrm{min}}\},[w'])$.
\begin{itemize}
    \item If there is $j\in A_{[w']}(U\cup \{i_{\mathrm{min}}\})$, then $(j,i_{\mathrm{min}})\notin E$ and hence $j\in A_{[w]}(U)$. By the minimality we get $i_{\mathrm{min}}\leq j$. Assume that $i_{\mathrm{min}}=j$, then $\{i_{\mathrm{min}}^2\}\cup U\in \mathcal{I}(\mathcal{G})$ and hence $i_{\mathrm{min}}\in I_1^{\mathrm{it}}$. This contradicts $[w]\in \mathcal{M}'(\mathcal{G})$ and every element of $A_{[w']}(U\cup \{i_{\mathrm{min}}\})$ is in fact strictly greater than $i_{\mathrm{min}}$. \vspace{0,1cm}
    \item Now choose $j\in U \cup \{i_{\mathrm{min}}\} $ in the complement. Then we have either $j\in K$ or $(j,w')\in E$. Both cases imply, again with the minimality of $i_{\mathrm{min}}$, that $i_{\mathrm{min}}\leq j$.
\end{itemize}
Now assume $i_{\mathrm{min}}\notin A_{[w]}(U)$. In this case we have to show that each element of the set $L(U\setminus \{i_{\mathrm{min}}\},[i_{\mathrm{min}}w])$ is greater or equal to $i_{\mathrm{min}}$. This is enough since $i_{\mathrm{min}}\in A_{[i_{\mathrm{min}}w]}(U\setminus \{i_{\mathrm{min}}\})$.
\begin{itemize}
    \item If $j\in A_{[i_{\mathrm{min}}w]}(U\setminus \{i_{\mathrm{min}}\})$, then $j=i_{\mathrm{min}}$ or $U\cup \{j\}\in \mathcal{I}(\mathcal{G})$ which gives $j\in A_{[w]}(U)$ and the minimality gives the claim. \vspace{0,1cm}
    \item If $j\in U\setminus \{i_{\mathrm{min}}\}$ lies in the complement we either have $j\in K$ or $(j,i_{\mathrm{min}}w')\in E$. In the first case we are done and also in the second case provided $(j,w')\in E$. However, if $(j,i_{\mathrm{min}})\in E$ we get a contradiction to $j,i_{\mathrm{min}}\in U$ and $U\in \mathcal{I}(\mathcal{G})$.
\end{itemize}
This proves the claim. Hence $\sigma$ is an involution from $M_V$ to $M_V$ and all $(U,[w])\in M_V$ in the summation \ref{inversion} cancel each other. The only remaining summands correspond to $(U,[w])$ for which $L(U,[w])=\emptyset$, thus the result follows from \eqref{well}.
\end{proof}
\subsection{}\label{subsect28} The partially commutative Lie superalgebra $\mathcal{P}(\mathcal{G})$ is $\mathbb{Z}_+^{I}$-graded (compatible with the $\mathbb{Z}_2$ grading) with finite-dimensional graded pieces, where the degree of $e_i\in X_I$ is the $i$-th unit vector $\alpha_i$. This induces a $\mathbb{Z}_+^{I}$-grading on $\mathbf{U}_{\mathcal{G}}=\mathbf{U}(\mathcal{P}(\mathcal{G}))$ in the obvious way and the Hilbert series is defined by
 $$\mathcal{H}_{\mathbf{U}_\mathcal{G}}(\mathbf{x})=\sum_{\mu\in \mathbb{Z}_+^{I}}\dim{\mathbf{U}(\mathcal{P}(\mathcal{G}))_\mu} \ x^{\mu}.$$
 
Set 
$$\mathrm{mult}_{\mathcal{G}}(\beta)=\dim \mathcal{P}(\mathcal{G})_{\beta},\ \ \Delta^{\mathcal{G}}=\{\beta\in \mathbb{Z}_+^{I}: \mathrm{mult}_{\mathcal{G}}(\beta)\neq 0\}=\Delta_0^{\mathcal{G}}\cup \Delta_1^{\mathcal{G}}$$
where $\Delta_j^{\mathcal{G}}$ is the subset of all $\beta\in \mathbb{Z}_+^{I}$ satisfying $\sum_{i\in I_1} \beta_i\equiv j \mod 2$. 
The set $\Delta^{\mathcal{G}}$ is called the set of roots. Choose a basis $\{x_{r,\beta}: 1\leq r\leq k_{\beta}\}$ of $\mathcal{P}(\mathcal{G})_{\beta}$ for all $\beta\in \Delta^{\mathcal{G}}$ and an ordering on the set $\Delta^{\mathcal{G}}=\{\beta_1,\beta_2,\dots\}$. Consider the following ordered basis of $\mathcal{P}(\mathcal{G})$:
$$\{x_{1,\beta_1},x_{2,\beta_1},\dots,x_{k_{\beta_1},\beta_1},x_{1,\beta_2},x_{2,\beta_2},\dots\}$$ 
The following expression is a simple consequence of the PBW Theorem: 
\begin{align*}\mathcal{H}_{\mathbf{U}_\mathcal{G}}(\mathbf{x})&=\prod_{\beta\in \Delta_1^{\mathcal{G}}}(1+x^\beta)^{\mathrm{mult}_{\mathcal{G}}(\beta)}\prod_{\beta\in \Delta_0^{\mathcal{G}}}(1+x^\beta+x^{2\beta}+\cdots)^{\mathrm{mult}_{\mathcal{G}}(\beta)}&\\&
=\frac{\prod_{\beta\in \Delta_1^{\mathcal{G}}}(1+x^\beta)^{\mathrm{mult}_{\mathcal{G}}(\beta)}}{\prod_{\beta\in \Delta_0^{\mathcal{G}}}(1-x^\beta)^{\mathrm{mult}_{\mathcal{G}}(\beta)}}   \end{align*}

The next proposition is derived from the above discussion, Lemma~\ref{invlem} and Proposition~\ref{monoid basis} which is the denominator identity for partially commutative Lie superalgebras. Recall that it can be also derived using certain identifications from \cite[Section 2.5]{Wa01}; here we presented a direct combinatorial approach.
\begin{prop}\label{denomi}
    We have 
    $$\mathcal{I}(\mathcal{G},-\mathbf{x})^{-1}=\frac{\prod_{\beta\in \Delta_1^{\mathcal{G}}}(1+x^\beta)^{\mathrm{mult}_{\mathcal{G}}(\beta)}}{\prod_{\beta\in \Delta_0^{\mathcal{G}}}(1-x^\beta)^{\mathrm{mult}_{\mathcal{G}}(\beta)}}$$
    \qed
\end{prop}
This will be used later to give a formula for $\mathrm{mult}_{\mathcal{G}}(\beta)$ and also to determine $\Delta^{\mathcal{G}}$.
\section{Marked multi-colorings and marked chromatic polynomials}\label{markedcolorings}
\subsection{}
Throughout this subsection let $\bold{m}:=(m_i:i\in I)$ be a tuple of non-negative integers with finite support.
\begin{defn} Given $q\in\mathbb{N}$ we define the following.
\begin{enumerate}
        \item We call a map $_{\bold{m}}\Gamma^{\mathrm{mark}}_\mathcal{G}:I\rightarrow P^{\mathrm{mult}}(\{1,\dots,q\})$ a marked multi-coloring of $\mathcal{G}$ associated to $\bold{m}$ using at most $q$-colors if the following conditions are satisfied:\vspace{0,1cm}
        
        \begin{enumerate}[(i)]
            \item $_{\bold{m}}\Gamma^{\mathrm{mark}}_\mathcal{G}(i)$ is a subset of $\{1,\dots,q\}$ for all $i\in I\backslash I_1^{\mathrm{it}}$,\vspace{0,1cm}
            \item for all $i\in I$ we have $|_{\bold{m}}\Gamma^{\mathrm{mark}}_\mathcal{G}(i)|=m_i$,\vspace{0,1cm}
            \item for all $i,j\in I$ such that $(i,j)\in E$, we have $_{\bold{m}}\Gamma^{\mathrm{mark}}_\mathcal{G}(i)\cap {_{\bold{m}}\Gamma^{\mathrm{mark}}_\mathcal{G}(j)}=\emptyset$.
            
        \end{enumerate}
       \vspace{0,1cm}
       
        \item The number of marked multi-colorings of $\mathcal{G}$ associated to $\bold{m}$  using at most $q$-colors 
        is denoted by $_{\bold{m}}\Pi^{\mathrm{mark}}_\mathcal{G}(q)$.
    \end{enumerate}
\end{defn}
As a first step, we will show as for ordinary chromatic polynomials that $_{\bold{m}}\Pi^{\mathrm{mark}}_\mathcal{G}(q)$ is a polynomial in $q$; in the ordinary case this follows from a recurrence relation called the deletion–contraction recurrence; see \cite{B93} for an overview of the classical results. This will allow to give a meaning for $_{\bold{m}}\Pi^{\mathrm{mark}}_\mathcal{G}(q)$ for any integer $q\in \mathbb{Z}$.
We denote by $P_k(\bold{m},\mathcal{G})$ the set of all ordered $k$-tuples $(P_1,\dots,P_k)$ satisfying:
\begin{enumerate}
\item $P_r$ is a non-empty multi-set and $P_r\in \mathcal{I}({\mathcal{G}})$ for all $1\leq k\leq r$ \medskip
    \item the disjoint union (as multisets) of $P_1,\dots, P_k$ is equal to the multiset $\{i^{m_i} : i\in\mathrm{supp}(\bold{m})\},$ i.e., $i$ appears exactly $m_i$ times for each $i\in \mathrm{supp}(\bold{m})$.
\end{enumerate}
Since $\mathrm{supp}(\bold{m})$ is finite, $P_k(\bold{m},\mathcal{G})$ is also finite. The following is the analogue result of \cite[Section 3.3]{AKV18} for marked multi-colorings.

\begin{prop}\label{props6}
For each $q\in \mathbb{N}$ and $\mathbf{m}\in\mathbb{Z}_+^I$ with non-empty support, we have
    \begin{equation}
    \label{equation1}
    _{\bold{m}}\Pi^{\mathrm{mark}}_\mathcal{G}(q)=\sum_{k\geq 1}| P_k(\mathbf{m},\mathcal{G})| \binom{q}{k}.
\end{equation} In particular, the co-efficient of $q$ in $_{\bold{m}}\Pi^{\mathrm{mark}}_\mathcal{G}(q)$ is $\sum_{k\geq 1}\frac{(-1)^{k-1}}{k!} |P_k(\mathbf{m},\mathcal{G})|.$
\end{prop}
\begin{proof}
We fix a subset $\{c_1<\dots<c_k\}\subseteq \{1,\dots,q\}$ of colors and consider marked multi-colorings $_{\bold{m}}\Gamma^{\mathrm{mark}}_\mathcal{G}$ which use only the $c_i's$ to color the vertices of $\mathcal{G}$. These marked multi-colorings are in natural bijection with $P_k(\mathbf{m},\mathcal{G})$ which finishes the proof by varying the number of colors and the exact choice of colors. To see this, let $_{\bold{m}}\Gamma^{\mathrm{mark}}_\mathcal{G}$ be a marked multi-coloring of $\mathcal{G}$ using only the aforementioned colors. Now we can associate a natural element $(P_1,\dots,P_k)\in {P}_k(\mathbf{m},\mathcal{G})$ as follows. Denoting by $a_{i,r}$ the multiplicity of color $c_r$ in $_{\bold{m}}\Gamma^{\mathrm{mark}}_\mathcal{G}(i)$, we set 
$$P_r=\{i^{a_{i,r}}: i\in I\},\ \ 1\le r\le k.$$
It follows from the definition of $_{\bold{m}}\Gamma^{\mathrm{mark}}_\mathcal{G}$ that $P_r\in \mathcal{I}({\mathcal{G}})$ is non-empty and these multi-sets $P_r,\ 1\le r\le k,$ form a partition of the multi-set $\{i^{m_i}: i\in \mathrm{supp}(\mathbf{m})\}$. Thus $(P_1,\dots,P_k)\in {P}_k(\mathbf{m},\mathcal{G})$. Conversely, given any multi-partition $(P_1,\dots,P_k)\in {P}_k(\mathbf{m},\mathcal{G})$ we
assign the color $c_r$ to the vertices in $P_r$ counted with multiplicity, for each $1\le r\le k$. This gives a marked multi-coloring of $\mathcal{G}$ and the afore described correspondence is bijective. 
\end{proof}
By the above proposition, we can treat $_{\bold{m}}\Pi^{\mathrm{mark}}_\mathcal{G}(q)$ as a polynomial in a formal variable $q$ and specialize $q$ to any integer. This polynomial is called the marked chromatic polynomial of $\mathcal{G}$ associated to $\mathbf{m}$.
\subsection{} Now, we establish the connection between the marked multivariate independence series of $\mathcal{G}$ and marked chromatic polynomials. We have the following relationship. 
\begin{thm}\label{expmarkedindp}
    Let $\mathcal{G}$ be a marked graph with vertex set $I$. For $q\in \mathbb{Z}$, we have as formal series
    $$\mathcal{I}(\mathcal{G},\mathbf{x})^q=\sum_{\mathbf{m}\in \mathbb{Z}_+^{I}} {_{\bold{m}}\Pi^{\mathrm{mark}}_\mathcal{G}(q)}\ x^\mathbf{m}.$$
\end{thm} 
\begin{proof} 
    For any $f\in R$ with constant term $1$, we have 
    $$f^q=\sum_{k\geq 0}\binom{q}{k}(f-1)^k,$$
    where we understand $\binom{-n}{k}=(-1)^k\binom{n+k-1}{k}$ for $n\in\mathbb{N}$.
    Set $f=\mathcal{I}(\mathcal{G},\mathbf{x})$ and note that 
    $$f-1 = \mathcal{I}(\mathcal{G},\mathbf{x})-1=\sum_{U\in \mathcal{I}(\mathcal{G})\setminus \{\emptyset\}} \prod_{i\in U}x_i.$$
    Thus, for $k\geq 1$ and $\mathbf{m}\neq 0$, the coefficient $(f-1)^k[x^\mathbf{m}]$ is given by $|P_k(\mathbf{m},\mathcal{G})|$ and we get with Proposition~\ref{props6}:
    
 
    $$\mathcal{I}(\mathcal{G}\mathbf{x})^q[x^\mathbf{m}]=\sum_{k\geq 1}\binom{q}{k}(f-1)^k[x^\mathbf{m}]=\sum_{k\geq 1}\binom{q}{k}|P_k(\mathbf{m},\mathcal{G})|={_{\bold{m}}\Pi^{\mathrm{mark}}_\mathcal{G}(q)},\  \text{for all $\mathbf{m}\neq 0$}.$$ This completes the proof; the coefficients for $\mathbf{m}=0$ clearly coincide.
\end{proof}

\subsection{}
There is also a connection between the ordinary chromatic polynomials and marked chromatic polynomials. We define the join of $\mathcal{G}$ with respect to $\bold{m}$, denoted by $\mathcal{G}(\bold{m})$, as follows:
the vertices of $\mathcal{G}(\bold{m})$ are $\{i_1,\dots,i_{m_i}: i\in \mathrm{supp}(\bold{m})\}$ and the edges of $\mathcal{G}(\bold{m})$ are given by $e(i_r,i_s)$ for all $1\leq r\neq s\leq m_i$ and $e(i_r,j_s)$ for each $e(i,j)\in E(\mathcal{G})$ and $1\leq r\leq m_i$ and $1\leq s\leq m_j$.  That is, we replace each vertex $i\in  \mathrm{supp}(\bold{m})$ by a clique of size $m_i$ and take the connected sum of the $i$-th and $j$-th cliques whenever $i$ and $j$ are connected in the graph $\mathcal{G}$. 

Let $m$ be a non-negative integer. A partition $\lambda$ of $m$, denoted as $\lambda \vdash m$, is an ordered tuple $\lambda = (\lambda_1\ge \cdots \ge \lambda_k>0)$ of non-negative integers such that $\sum_{i=1}^k \lambda_i =m$, and the length of the partition $\ell(\lambda)$ is defined to be $k$. For any $j\in \mathbb{N}$, let $d^\lambda_{j}$ be the number appearances of $j$ in the partition $\lambda$. We further define 
$$S(\bold m)= \left\{\boldsymbol{\lambda} = (\boldsymbol{\lambda}_i)_{i\in I} : \boldsymbol{\lambda}_i \vdash m_i\ \forall i\in I, \text{ and } \boldsymbol{\lambda}_i=(1^{m_i})\ \forall i\notin I_1^{\mathrm{it}}\right\}.$$ 
Note that $\boldsymbol{\lambda}_i$ is the empty partition for all $i\in I$ with $m_i=0$.
Moreover, for $\boldsymbol{\lambda}\in S(\bold m)$, we set $ \bold s(\boldsymbol{\lambda})=(\ell(\boldsymbol{\lambda}_i) : i\in I)$, where the length of the empty partition is understood to be zero.
We have the following relation between the ordinary chromatic polynomials (denoted by $\pi_{\mathcal{G}}(q)$) and marked chromatic polynomials. 
\begin{prop}\label{ordinarychormatic} We have 
   \begin{equation}
    \label{equation}
    {_{\bold{m}}\Pi^{\mathrm{mark}}_\mathcal{G}(q)}=\sum\limits_{\boldsymbol{\lambda}\in S(\bold m)} \frac{\pi_{\mathcal{G}(\bold s(\boldsymbol{\lambda}))}(q)}{\prod_{i\in \mathrm{supp}(\bold m)}\prod_{k=1}^{\infty}(d^{\boldsymbol{\lambda}_i}_{k}!)}
\end{equation} 
\end{prop}
\begin{proof}
     Let ${_{\bold{m}}\Gamma^{\mathrm{mark}}_\mathcal{G}}$ be a marked multi-coloring of $\mathcal{G}$ associated to $\bold{m}$ using at most $q$ colors.
     By definition, $_{\bold{m}}\Gamma^{\mathrm{mark}}_\mathcal{G}(i)$ is a multi-subset (resp. subset) of $\{1,\dots,q\}$ if $i\in I_1^{\mathrm{it}}$ (resp. $i\notin I_1^{\mathrm{it}}$)  with $|{_{\bold{m}}\Gamma^{\mathrm{mark}}_\mathcal{G}}(i)|=m_i$. We associate an element $\boldsymbol{\lambda} = (\boldsymbol{\lambda}_i)_{i\in I}\in S(\bold m)$ and a family of ordinary vertex colorings of $\mathcal{G}(\bold s(\boldsymbol{\lambda}))$ as follows. If $m_i=0$, then $\boldsymbol{\lambda}_i=\emptyset$. Otherwise $i\in \mathrm{supp}(\mathbf{m})$ and let 
     $${_{\bold{m}}\Gamma^{\mathrm{mark}}_\mathcal{G}}(i)=\{1^{b_{i,1}},\dots,q^{b_{i,q}}\}.$$
     Then $\boldsymbol{\lambda}_i=(\lambda^1_i\ge \cdots \ge \lambda^{\ell(\boldsymbol{\lambda}_i)}_i>0)$ is defined to be the unique partition of $m_i$ associated to the set $\{b_{i,1},\dots,b_{i,q}\}$. Now consider the graph $\mathcal{G}(\bold s(\boldsymbol{\lambda}))$ and fix an ordering 
     on the vertices of each clique (recall that the size of the cliques are given by $\ell(\boldsymbol{\lambda}_i)$). For any choice of colors, $c_1,\dots,c_{\ell(\boldsymbol{\lambda}_i)}$ with
\begin{equation}\label{4re}\lambda^1_i=b_{i,c_1},\dots,\lambda^{\ell(\boldsymbol{\lambda}_i)}_i=b_{i,c_{\ell(\boldsymbol{\lambda}_i)}}\end{equation}
we color the first vertex of the $i$-th clique of $\mathcal{G}(\bold s(\boldsymbol{\lambda}))$ by $c_1$, the second vertex with $c_2$ and so on. This gives an ordinary vertex coloring. It is easy to see that the number of choices of colors satisfying \eqref{4re} and thus the number of ordinary colorings we get is exactly $\prod_{i\in \mathrm{supp}(\bold m)}\prod_{k=1}^{\infty}(d^{\boldsymbol{\lambda}_i}_{k})!$

Now conversely let $\boldsymbol{\lambda} = (\boldsymbol{\lambda}_i)_{i\in I}\in S(\bold m)$ and fix an ordering on the vertices of $\mathcal{G}(\bold s(\boldsymbol{\lambda}))$. Let ${\tau}_{\mathcal{G}(\bold s(\boldsymbol{\lambda}))}$ be a vertex coloring  of $\mathcal{G}(\bold s(\boldsymbol{\lambda}))$ and suppose that the first vertex of the $i$-th clique receives color $c_1$, the second $c_2$ and so on.
We associate a marked multi-coloring $_{\bold{m}}\Gamma^{\mathrm{mark}}_\mathcal{G}$ of $\mathcal{G}$ associated to $\bold{m}$ as follows. Define 
$${_{\bold{m}}\Gamma^{\mathrm{mark}}_\mathcal{G}}(i)=\{c_1^{\lambda^1_i},\dots,c_r^{\lambda^r_i}\},\ \text{ where $\boldsymbol{\lambda}_i=(\lambda^1_i\ge \cdots \ge \lambda^r_i>0)$}.$$
Again, if $\lambda_{i}^{r_1}=\lambda_{i}^{r_2}$, then the vertex coloring obtained from ${\tau}_{\mathcal{G}(\bold s(\boldsymbol{\lambda}))}$ by interchanging the colors $c_{r_1}$ and $c_{r_2}$ in the $i$-th clique gives the same marked multi-coloring. Hence $\prod_{k=1}^{\infty}(d^{\boldsymbol{\lambda_i}}_{k}!)$ number of usual vertex coloring of $\mathcal{G}(\bold s(\boldsymbol{\lambda}))$ correspond to exactly one marked multi-coloring of $\mathcal{G}$ associated to $\bold{m}$. Taking the sum over all such tuples of partitions gives ${_{\bold{m}}\Pi^{\mathrm{mark}}_\mathcal{G}}(q)$, and the claim follows.
\end{proof}
\begin{example} We compute the marked chromatic polynomial of the following marked star graph. We emphasize that $2,n-1,n$ (resp. $1,3,n-2$) are fixed to be white (resp. black) and the nodes between $4$ and $n-3$ can be arbitrary: 
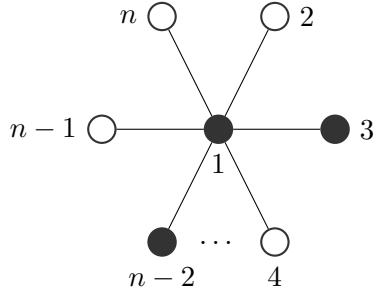
\begin{figure}[H]
\centering
\begin{subfigure}{0.49\textwidth}
\centering
\begin{tikzpicture}
        \tikzstyle{B}=[circle,draw=black!80,thick]
        \tikzstyle{C}=[circle,draw=black!80,fill=black!80,thick]
        \node[C] (1) at (0,0) [label=below:$n-2$]{};
        \node[B] (2) at (1.5,0) [label=below:$4$]{};
        \node[B] (3) at (-0.8,1.5) [label=left:$n-1$]{};
        \node[C] (4) at (2.3,1.5) [label=right:$3$]{};
        \node[B] (5) at (1.5,3) [label=right:$2$]{};
        \node[C] (6) at (0.75,1.5) [label=below:$1$]{};
        \node[B] (7) at (0,3) [label=left:$n$]{};
        \path[-] (1) edge node[left]{} (6);
        \path[-] (2) edge node[left]{} (6);
        \path[-] (3) edge node[left]{} (6);
        \path[-] (4) edge node[left]{} (6);
        \path[-] (5) edge node[left]{} (6);
        \path[-] (7) edge node[left]{} (6);
        \path (1) -- node[auto=false]{\ldots} (2);
    \end{tikzpicture}
\end{subfigure}
\end{figure}
First we will consider the colorings corresponding to $\boldsymbol{\lambda}\in S(\bold m)$, i.e. marked multi-colorings where $_{\bold{m}}\Gamma^{\mathrm{mark}}_\mathcal{G}(i)$, $i\in\mathrm{supp}(\mathbf{m})$, is of the form 
\begin{equation}\label{form}{_{\bold{m}}\Gamma^{\mathrm{mark}}_\mathcal{G}}(i)=\{c_1^{\lambda^1_i},\dots,c_{r_i}^{\lambda^{r_i}_i}\},\ \text{ $\boldsymbol{\lambda}_i=(\lambda^1_i\ge \cdots \ge \lambda^{r_i}_i>0)$}\end{equation}
for some $c_1,\dots,c_{r_i}\in\{1,\dots,q\}.$
Clearly, $_{\bold{m}}\Gamma^{\mathrm{mark}}_\mathcal{G}(1)$ is a subset of $\{1,\dots,q\}$, so we have $\binom{q}{m_1}$ choices to color vertex $1$. We continue with vertex $2\in I_1^{\mathrm{it}}$. If $m_2\neq 0$, 
we can color vertex $2$ with $r_2=\ell(\boldsymbol{\lambda}_2)$ distinct colors (recall that the multiplicities are fixed by the partition) from the remaining $q-m_1$ colors. So the number of possible colorings of vertex 2 is given by $\binom{q-m_1}{r_2}\frac{r_2!}{\prod_{k=1}^{\infty}(d_k^{\boldsymbol{\lambda}_2}!)}$. 
Since $3\notin I_1^{\mathrm{it}}$, to color vertex $3$ we can choose any $m_3$ colors with the restriction that we can not use the colors used for vertex $1$. Hence we have $\binom{q-m_1}{m_3}$ choices. We continue coloring the remaining vertices in the same fashion depending whether a vertex is contained in $\mathrm{supp}(\mathbf{m})\cap I_1^{\mathrm{it}}$ or not, because all of those vertices are only adjacent to vertex $1$.
Thus the number of coloring corresponding to $\boldsymbol{\lambda}$ is given by
$$\binom{q}{m_1}\prod_{j\in \mathrm{supp}(\mathbf{m}), j\neq 1}\binom{q-m_1}{\ell(\boldsymbol{\lambda}_j)}\frac{\ell(\boldsymbol{\lambda}_j)!}{\prod_{k=1}^{\infty}(d^{\boldsymbol{\lambda}_{j}}_k!)}$$ and hence for the marked star graph we get
$${_{\bold{m}}\Pi^{\mathrm{mark}}_\mathcal{G}}(q)=\binom{q}{m_1}\sum_{\boldsymbol{\lambda}\in S(\bold m)}\prod_{j\in \mathrm{supp}(\mathbf{m}), j\neq 1}\binom{q-m_1}{\ell(\boldsymbol{\lambda}_j)}\frac{\ell(\boldsymbol{\lambda}_j)!}{\prod_{k=1}^{\infty}(d^{\boldsymbol{\lambda}_{j}}_k!)}.$$
\end{example}
\subsection{}\label{peographsection} We continue determining the marked chromatic polynomials for chordal graphs. We first define what a prefect elimination ordering is. 
\begin{defn} A total ordering $\leq $ of the vertices of $\mathcal{G}$ is called a perfect elimination ordering if for each $k\in I$, the subgraph $\mathcal{G}_k$ induced by the set of vertices 
 $$\{r\in I\setminus \{k\}:e(r,k)\in E, r<k\}\cup \{k\}$$ 
 is a finite complete subgraph. A graph $\mathcal{G}$ is called a PEO-graph if there exists a perfect elimination ordering on its vertices.

\end{defn}
\begin{thm}\label{peomarked}
    Let $\mathcal{G}$ be a PEO-graph with a countable vertex set $I$. We have
     \begin{equation}
    {_{\bold{m}}\Pi^{\mathrm{mark}}_\mathcal{G}}(q)=\sum\limits_{\boldsymbol{\lambda}\in S(\bold m)}\prod _{j\in \mathrm{supp}(\bold{m})} \binom{q-b^{\boldsymbol{\lambda}}_{j}}{\ell(\boldsymbol{\lambda}_{j})} \frac{\ell(\boldsymbol{\lambda}_{j})!}{\prod_{k=1}^{\infty}(d^{\boldsymbol{\lambda}_{j}}_{k}!)},\ \ b^{\boldsymbol{\lambda}}_{j}=\sum\limits_{\substack{i\in \mathcal{G}_j\backslash\{j\}}}\ell(\boldsymbol{\lambda}_{i})
        \end{equation}

\end{thm}
\begin{proof}
     Let $\mathrm{supp}(\bold{m})=\{i_1<\dots<i_N\}$ with respect to the chosen PEO-ordering and note that it is enough to consider the subgraph corresponding to the nodes in $\mathrm{supp}(\bold{m})$ which we relabel as $\{1,\dots,N\}$. For simplicity, we call this subgraph in the rest of the proof also by $\mathcal{G}$. Clearly, this is again a PEO-graph with the obvious induced ordering. As above, we will consider colorings corresponding to $\boldsymbol{\lambda}\in S(\bold m)$, i.e. those of the form \eqref{form}. To color the first vertex, we have as above (recall that if $1\notin I_1^{\mathrm{it}}$, then $\boldsymbol{\lambda}_1=(1^{m_1})$) 
     $$\binom{q}{\ell(\boldsymbol{\lambda}_{1})}\frac{\ell(\boldsymbol{\lambda}_{1})!}{\prod_{k=1}^{\infty}(d^{\boldsymbol{\lambda}_{1}}_k!)}$$ possibilities. 
   Now suppose that we have colored the first $n-1$ vertices, $2\leq n<N$ and note that the subgraph $\mathcal{G}_n$ is complete. So in order to color vertex $n$ we can not use the colors that are used to color the vertices in $\mathcal{G}_n\backslash\{n\}$. So we have to choose $\ell(\boldsymbol{\lambda}_n)$ colors out of the remaining $q-b^{\boldsymbol{\lambda}}_{n}$
   and we get $$\binom{q-b^{\boldsymbol{\lambda}}_{n}}{\ell(\boldsymbol{\lambda}_n)}\frac{\ell(\boldsymbol{\lambda}_n)!}{\prod_{k=1}^{\infty}(d^{\boldsymbol{\lambda}_{n}}_k!)}$$
   choices to color vertex $n$. Taking the sum over all such partitions $\boldsymbol{\lambda}\in S(\bold m)$, gives the result.

\end{proof}
A chordal graph is one in which all cycles of four or more vertices have an edge that is not part of the cycle but connects two vertices of the cycle. It is known that a graph is chordal if and only if it has a PEO-ordering. For example, complete graphs $K_n$ are chordal and any ordering gives a PEO-ordering.
\section{Root multiplicities, BKM superalgebras and marked chromatic polynomials}\label{rootmultisection}
\subsection{}
In this subsection, we connect the root multiplicities of $\mathcal{P}(\mathcal{G})$ with the marked chromatic polynomials. We denote by ${_{\bold{m}}\Pi^{\mathrm{mark}}_\mathcal{G}}(q)[q]$ the coefficient of $q$ in
${_{\bold{m}}\Pi^{\mathrm{mark}}_\mathcal{G}}(q)$ and let $\mu:\mathbb{N}\rightarrow \{0,\pm1\}$ be the M\"{o}bius function. Given $\bold m = (m_i : i\in I)\in \mathbb{Z}_{+}^I$ and $\ell \ge 1$, we say
$\ell | \bold m$ if $\ell | m_i$ for each $i\in I$ and set $\epsilon(\bold{m})=1$ if the sum over its odd entries is even and $\epsilon(\bold{m})=-1$ otherwise. The following result generalizes \cite[Corollary 3.9]{AKV18} and \cite[Corollary 1.4.]{SA2024}.
\begin{thm}\label{multroo1} For any $\bold m\in \mathbb{Z}_+^{I}$, we have
$$ \mathrm{mult}_{\mathcal{G}}(\bold m) = \epsilon\big(\bold m)\sum\limits_{\ell | \bold m}\frac{\mu(\ell)}{\ell}\epsilon\big(\bold m/\ell\big)(-1)^{|\bold m|/\ell-1} {_{\frac{\bold{m}}{\ell}}\Pi^{\mathrm{mark}}_\mathcal{G}}(q)[q].$$
In particular, if the $m_i$'s are relatively prime, then  $\mathrm{mult}_{\mathcal{G}}(\bold m)=|{_{\bold{m}}\Pi^{\mathrm{mark}}_\mathcal{G}}(q)[q]|$.         
\end{thm}
\begin{proof} From Proposition~\ref{denomi} we get 
    $$\mathcal{I}(\mathcal{G},-\mathbf{x})=\prod_{\beta\in \Delta^{\mathcal{G}}}(1-\epsilon(\beta)x^\beta)^{\epsilon(\beta)\mathrm{mult}_{\mathcal{G}}(\beta)}.$$
Taking $-\mathrm{log}$ on both sides of this identity, we obtain 
 \begin{equation*} \label{eq1}
\begin{split}
-\mathrm{log}\ \mathcal{I}(\mathcal{G},-\mathbf{x}) & = \sum\limits_{\beta\in \Delta^{\mathcal{G}}}\epsilon(\beta)\mathrm{mult}_{\mathcal{G}}(\beta)(-\mathrm{log}(1-\epsilon(\beta)x^\beta)) \\
 & = \sum\limits_{\beta\in \Delta^{\mathcal{G}}}\sum\limits_{k\ge1}\epsilon(\beta)^{k+1}\mathrm{mult}_{\mathcal{G}}(\beta)\frac{x^{k\beta}}{k}.
\end{split}
\end{equation*} 
  Hence, for any $\bold m = (m_i : i\in I)\in\mathbb{Z}_{+}^{I}$ we conclude
   $$-\mathrm{log}\ \mathcal{I}(\mathcal{G},-\mathbf{x})[x^{\bold m}]= \sum\limits_{k|\bold m}\frac{1}{k}\epsilon\big(\bold m/k\big)^{k+1}\mathrm{mult}_{\mathcal{G}}(\bold m/k). $$
On the other hand, we have from the Propositions \ref{props6} and the definition of $P_r(\mathbf{m},\mathcal{G})$:
$$-\mathrm{log}\ \mathcal{I}(\mathcal{G},-x)[x^{\bold m}]=(-1)^{|\bold m|}\sum_{r\geq 1}\frac{(-1)^{r}}{r}|P_r(\mathbf{m},\mathcal{G})|= (-1)^{|\bold m|-1}{_{\bold{m}}\Pi^{\mathrm{mark}}_\mathcal{G}}(q)[q].$$
Therefore, using $\epsilon(\mathbf{m})\epsilon(\mathbf{m}/k)^{k+1}=\epsilon(\mathbf{m}/k)$ for any $k|\mathbf{m}$ we get
$$ \sum\limits_{k|\bold m}\frac{|\mathbf{m}|}{k}\epsilon\big(\bold m/k\big)\mathrm{mult}_{\mathcal{G}}(\bold m/k)  = (-1)^{|\bold m|-1}|\mathbf{m}|\epsilon(\mathbf{m})\ {_{\bold{m}} \Pi^{\mathrm{mark}}_\mathcal{G}}(q)[q].$$
Now the M\"{o}bius inversion formula gives
$$\mathrm{mult}_{\mathcal{G}}(\bold m) = \epsilon\big(\bold m)\sum\limits_{\ell | \bold m}\frac{\mu(\ell)}{\ell}\epsilon\big(\bold m/\ell\big)(-1)^{|\bold m|/\ell-1} {_{\frac{\bold{m}}{\ell}}\Pi^{\mathrm{mark}}_\mathcal{G}}(q)[q].$$
\end{proof}
\begin{example}\label{exchrhalb}\begin{enumerate}\item Let $i\in I_1^{\mathrm{it}}$ and $j\in I$ arbitrary. From the super Jacobi identity and $[e_i,e_i]=0$ we get $[e_i,[e_i,e_j]]=0$ in $\mathcal{P}(\mathcal{G})$ and thus $\mathbf{m}$ with $m_i=2, m_j=1$ and zero otherwise is not a root. Alternatively, we can check this using Theorem~\ref{multroo1} and \eqref{equation}, namely 
$$\mathrm{mult}_{\mathcal{G}}(\bold m)=|{_{\bold{m}}\Pi^{\mathrm{mark}}_\mathcal{G}}(q)[q]|=|\pi_{\mathcal{G}'}(q)[q]+\frac{1}{2}\pi_{\mathcal{G}''}(q)[q]|=|-1+1|=0$$
where $\mathcal{G}'$ is a line with two nodes and $\mathcal{G}''$ is a triangle.
\item Let $\mathbf{m}\in \mathbb{Z}^I_+$ such that $m_i\in \{0,2\}$ for all $i\in I$ and suppose that the full subgraph spanned by $\mathrm{supp}(\mathbf{m})$ is a star graph satisfying $i,j\in I_1^{\mathrm{it}}$ whenever $(i,j)\notin E$. If $|\mathrm{supp}(\bold m)|\geq 3$, we have \begin{equation}\label{einhalb} _{\bold{m}}\Pi^{\mathrm{mark}}_\mathcal{G}(q)[q]=-\frac{1}{2}.\end{equation}
    This follows immediately from the calculation of the marked chromatic polynomial
    which is
    $${_{\bold{m}}\Pi^{\mathrm{mark}}_\mathcal{G}}(q)=\binom{q}{2}\binom{q-1}{2}^{|\mathrm{supp}(\bold m)|-1}+\epsilon q\binom{q}{2}^{|\mathrm{supp}(\bold m)|-1}$$
    where $\epsilon=1$ if the center of the star graph is odd isotropic and zero otherwise. We also have \eqref{einhalb} if $|\mathrm{supp}(\bold m)|=2$ and $\mathrm{supp}(\bold m)\nsubseteq I\backslash I_1^{\mathrm{it}}$.
    \end{enumerate}
\end{example}

\begin{rem}
Similarly, the marked chromatic polynomial can be expressed in terms of root multiplicities as follows: for all $\bold m\in \mathbb{Z}_+^{I}$ and $q\in\mathbb{Z}$, we have
$$
{_{\bold{m}}\Pi^{\mathrm{mark}}_\mathcal{G}}(q) = (-1)^{|\bold m|} \sum\limits_{\bold J} (-1)^{\sum k_i}
\prod_{i=1}^r\epsilon(\beta_i)^{k_i}\binom{q \epsilon(\beta_i)\mathrm{mult}_{\mathcal{G}}(\beta_i)}{k_i}.$$
where the sum runs over $$\bold J\in \dot\bigcup_{r\ge 1} \left\{((\beta_i, k_i))_{i=1}^r\in (\Delta^{\mathcal{G}}\times \mathbb{Z}_+)^{\times r}  :  k_1\beta_1+\cdots +k_r\beta_r = \bold m\right\}.$$
For the explicit description of $\Delta^{\mathcal{G}}$ see Proposition~\ref{rootdescp}. This generalizes \cite[Theorem 1.3]{SA2024} for free roots of partially commutative Lie superalgebras and \cite[Theorem 1]{AKV18} for partially commutative Lie algebras.
\end{rem}
\subsection{}\label{BKMconnections}
Here we recall the notion of Borcherds-Kac-Moody (BKM) superalgebras from \cite{Ray06,Wa01} and extract the relevant information needed for $\mathcal{P}(\mathcal{G})$. 
Fix a subset $\Psi\subseteq I$. Let $\lie h_{\mathbb{R}}$ be a real vector space with a non-degenerate symmetric real valued bilinear form $(.,.)$ and fix elements $\tilde{h}_i$, $i\in I$ such that
\begin{enumerate}
	\item $(\tilde{h}_i,\tilde{h}_j)\leq 0$ if $i\neq j$
	\item $(\tilde{h}_i,\tilde{h}_i)>0 \implies 2\frac{(\tilde{h}_i,\tilde{h}_j)}{(\tilde{h}_i,\tilde{h}_i)} \in\mathbb{Z}$ for all $j\in I$
	\item $(\tilde{h}_i,\tilde{h}_i)>0,\ i \in \Psi \implies \frac{(\tilde{h}_i,\tilde{h}_j)}{(\tilde{h}_i,\tilde{h}_i)} \in\mathbb{Z}$ for all $j\in I$.
    
\end{enumerate} 
We set $\lie h=\lie h_{\mathbb{R}}\otimes_{\mathbb{R}}\mathbb{C}$, $a_{ij}=(\tilde{h}_i,\tilde{h}_j)$ and 
$$I^{\mathrm{re}}=\{i\in I: a_{ii}>0\},\ \ I^{\mathrm{im}}=I\backslash I^{\mathrm{re}},\ \ \Psi^{re} = \Psi \cap I^{re},\ \ \Psi_0 = \{i \in \Psi^{} : a_{ii}= 0\}.$$ 
The symmetric real valued matrix $A=(a_{ij})$ is called a BKM supermatrix.

The BKM superalgebra $\lie g=\lie g(A,\lie h,\Psi)$ associated to $(A,\lie h, \Psi)$ is the Lie superalgebra generated by $\tilde{e}_i, \tilde{f}_i, i \in I$ and $\lie h$, where $\tilde{e}_i$ and $\tilde{f}_i$ are odd (resp. even) elements for $i\in \Psi$ (resp. $i \notin \Psi$) and all elements $\tilde{h}_i, i\in I$ are even, with the following defining relations:
\begin{enumerate}
	\item $[h,h']=0$ for $h,h'\in \lie h$,\vspace{0,1cm}
	\item $[h, \tilde{e}_j]=(h,\tilde{h}_j)\tilde{e}_j$,  $[h, \tilde{f}_j]=-(h,\tilde{h}_j)\tilde{f}_j$ for $h\in \lie h, j\in I$,\vspace{0,1cm}
	\item $[\tilde{e}_i, \tilde{f}_j]=\delta_{ij}\tilde{h}_i$ for $i, j\in I$,\vspace{0,1cm}
	\item $(\text{ad }\tilde{e}_i)^{1-\frac{2a_{ij}}{a_{ii}}}\tilde{e}_j=0 = (\text{ad }\tilde{f}_i)^{1-\frac{2a_{ij}}{a_{ii}}}\tilde{f}_j$ if $i \in I^{re}$ and $i \ne j$, \vspace{0,1cm}
	\item $(\text{ad }\tilde{e}_i)^{1-\frac{a_{ij}}{a_{ii}}}\tilde{e}_j=0 = (\text{ad }\tilde{f}_i)^{1-\frac{a_{ij}}{a_{ii}}}\tilde{f}_j$ if $i \in \Psi^{\mathrm{re}}$ and $i \neq j$,\vspace{0,1cm}
	\item $[\tilde{e}_i, \tilde{e}_j]= 0 = [\tilde{f}_i, \tilde{f}_j]$ if $a_{ij}=0$. 
\end{enumerate}
We write $\lie g=\lie n^{-}\oplus \lie h \oplus \lie n^+$ as vector spaces, where $\lie n^{+}$ (resp. $\lie n^{-}$) is the Lie sub-superalgebra generated by the elements $\tilde{e}_i$ (resp. $\tilde{f}_i$),\ $i\in I$ (see \cite[Corollary 2.1.19]{Ray06}). Let $Q$ be the formal root lattice defined to be the free abelian group generated by $\tilde{\alpha}_i$, $i\in I$ and for $\alpha=\sum_{k=1}^j \tilde{\alpha}_{i_k}\in Q$ let $\lie g_{\alpha}$ (resp. $\lie g_{-\alpha}$) be the subspace of $\lie g$ spanned by the elements
$$[\tilde{e}_{i_j}[\cdots[\tilde{e}_{i_2},\tilde{e}_{i_1}]],\ \ \ (\text{resp. } \ [\tilde{f}_{i_j}[\cdots[\tilde{f}_{i_2},\tilde{f}_{i_1}]]).$$
Let $Q^+\subseteq Q$ be the subset of non-negative linear combinations of the $\tilde{\alpha}_i's$ and denote the set of roots by $\Delta(\lie g)=\{\alpha\in Q: \lie g_{\alpha}\neq 0\}$. We have (see \cite[Proposition 2.3.3/2.3.5]{Ray06})
$$\lie g_{\alpha_i}=\mathbb{C} \tilde{e}_i,\ \ \lie g_{-\alpha_i}=\mathbb{C} \tilde{f}_i,\ \ \lie g_{\alpha}\neq 0\implies \alpha\in Q^+ \text{ or }-\alpha\in Q^+,\ \mathrm{dim}(\lie g_{\alpha})=\mathrm{dim}(\lie g_{-\alpha}).$$ 
The notions of positive and negative roots, the height of a root, and the support of a root are defined in the obvious way.
\begin{rem}\begin{enumerate}
    \item If $\Psi=\emptyset$, then the BKM superalgebra is simply a Borcherds algebra. 
    \item  Given a BKM Lie superalgebra $\lie g$, let $\mathcal{G}$ be the marked simple graph with marking $I_1^{\mathrm{it}}\subseteq I_1\subseteq I$ determined by
    $$(i,j)\in E,\ \text{iff $a_{ij}\neq 0$,\  $i\neq j$},\ I_1=\Psi,\ \Psi_0=I_1^{\mathrm{it}}.$$
In \cite{SA2024}, free roots of BKM superalgebras have been studied. These are roots $\alpha=\sum_{i\in I}k_i\tilde{\alpha}_i\in Q^+$ with $k_i\leq 1$ for all $i\in I^{\mathrm{re}}\cup \Psi_0$. The root spaces $\lie g_{\alpha}$ corresponding to free roots are closely related to root spaces of partially commutative Lie superalgebras defined in \cite[Definition 3.3]{SA2024} (recall that their definition is different from ours). The same idea (c.f. \cite[Lemma 3.10]{SA2024}) extends to root spaces $\lie g_{\alpha}$ with the weaker restriction $k_i\leq 1$ for all $i\in I^{\mathrm{re}}$ and are identifed with $\mathcal{P}(\mathcal{G})_{\sum_{i\in I}k_i \alpha_i}$. Hence the study of $\mathcal{P}(\mathcal{G})$ is also important from the perspective of BKM Lie superalgebras.


\end{enumerate}
   
\end{rem}
\subsection{} The first part of the following lemma follows from \cite[Lemma 3.10 and Propositon 5.1]{SA2024}, while the second part can be derived from \cite[Lemma 2.3.31]{Ray06} by identifying $\mathcal{P}(\mathcal{G})$ with the positive part $\mathfrak{n}^+$ of a BKM superalgebra.
\begin{lem}\label{someroots}
\begin{enumerate}
    \item Any element $\mathbf{m}\in \mathbb{Z}^I_+$ with connected support and $m_i\leq 1$ for all $i\in I_1^{\mathrm{it}}$ is contained in  $\Delta^{\mathcal{G}}$.
    \item Let $\alpha,\beta\in \Delta^{\mathcal{G}}\backslash \{\alpha_i: i\in I_1^{\mathrm{it}}\}$ such that $\mathrm{supp}(\alpha+\beta)$ is connected and $\alpha,\beta$ are not proportional (i.e. $\alpha\neq \beta, \alpha\neq 2\beta,\beta\neq 2\alpha$). Then $\alpha+\beta\in \Delta^{\mathcal{G}}.$ 
\end{enumerate}\qed
\end{lem}
Although the identification with BKM superalgebras has been established, the explicit description of $\Delta^{\mathcal{G}}$ remains unclear. The root description in \cite[Theorem 2.46]{Wa01} for $I^{\mathrm{re}}=\emptyset$ appears to be correct only under the additional assumption $\Psi_0=\emptyset$. We will provide a description of $\Delta^{\mathcal{G}}$ in the next subsection.

\section{Description of the set of roots of $\mathcal{P}(\mathcal{G})$}\label{rootssection}
Recall the definition of $\Delta^{\mathcal{G}}$ from Section~\ref{subsect28}. 

\subsection{}Let $N_\mathcal{G}(i)=\{j\in I: j\neq i, (i,j)\in E\}$. We call $\mathbf{m}\in \mathbb{Z}^I_+$ a \textit{star element} if the full subgraph of $\mathcal{G}$ spanned by $\mathrm{supp}(\mathbf{m})$ is a star graph.
\begin{prop}\label{rootdescp} Let $P$ be the set of elements $\mathbf{m}\in \mathbb{Z}^I_+$ satisfying 
\begin{itemize}
    \item $\mathrm{supp}(\mathbf{m})\ \text{is connected and $|\mathrm{supp}(\mathbf{m})|\geq 2$}$ \vspace{0,1cm}
    \item $\sum_{j\in N_\mathcal{G}(i)}m_j\geq m_i$ for all  $i\in \mathrm{supp}(\mathbf{m})\cap I_1^{\mathrm{it}}\ \text{with } m_i\geq 2$
\end{itemize}
and let $K^0_s$ (resp. $K^1_s$) be the subset of star elements $\mathbf{m}\in P$ satisfying 
\begin{itemize}
\item $m_i=m_j\geq s$ for all $i,j\in \mathrm{supp}(\mathbf{m})$ and $|\mathrm{supp}(\mathbf{m})\cap I_1|$ is even (resp. odd).
 \vspace{0,1cm}
    \item $i,j\in I_1^{\mathrm{it}}$ for all $i,j\in \mathrm{supp}(\mathbf{m})$ with $(i,j)\notin E$
    \vspace{0,1cm}
    \item $\mathrm{supp}(\mathbf{m})\cap I_1^{\mathrm{it}}\neq \emptyset$. \vspace{0,1cm}

\end{itemize}
Then we have \begin{equation}\label{rootsdes}\Delta^{\mathcal{G}}=(P\backslash K) \cup \{\alpha_i: i\in I\} \cup \{2\alpha_i: i\in I_1\backslash I_1^{\mathrm{it}}\},\ \ K:=K_2^0\cup K_3^1.\end{equation}
\end{prop}

\begin{proof} 
Assume that $\mathbf{m}\in \Delta^{\mathcal{G}}$, then $\mathrm{supp}(\mathbf{m})$ has to be connected; this is a standard proof in Lie theory. Moreover, if $|\mathrm{supp}(\mathbf{m})|=1$, then clearly $\mathbf{m}=\alpha_i$ for some $i\in I$ or $\mathbf{m}=2\alpha_i$ for some $i\in I_1\backslash I_1^{\mathrm{it}}$. So assume that $|\mathrm{supp}(\mathbf{m})|\geq 2$ and let $w=i_1\cdots i_r\in I^*$ with $e(w)\neq 0$ and $e(w)$ has grade $\mathbf{m}$. 
Then $w$ has the following properties for every $i\in \mathrm{supp}(\mathbf{m})\cap I_1^{\mathrm{it}}$ with $m_i\geq 2$:
\begin{itemize}
    \item We have $i_r\neq i$, since otherwise we can swap $i_{r-1}$ with $i_r$ and the corresponding Lie word is the same up to sign. Note that $i_{r-1}\neq i$ if $i_r=i$, since $[e_i,e_i]=0$. 
    
    \item Between two $i$'s in $w$ there has to be an index $j\in N_{\mathcal{G}}(i)$, since otherwise $e(w)=0$. This follows from $[e_i,[e_i,x]]=0$ and $[e_i,e_k]=0$ for all $k\notin N_{\mathcal{G}}(i)$. Similarly, to the right of the right most $i$ in $w$, there has to be an index in $N_{\mathcal{G}}(i)$.
\end{itemize}
Hence we must have $\sum_{j\in N_\mathcal{G}(i)}m_j\geq m_i$ which gives $\mathbf{m}\in P$. Assume by contradiction $\mathbf{m}\in K$ and write $\mathbf{m}=m \tilde{\mathbf{m}}$ for some $\tilde{\mathbf{m}}\in\mathbb{Z}_+^I$ with $\tilde{m}_{i}\in \{0,1\}$ and $m\geq 2$ (resp. $m\geq 3$) if $|\mathrm{supp}(\mathbf{m})\cap I_1|$ is even (resp. odd). Let $j_0\in \mathrm{supp}(\mathbf{m})$ be the center of the star graph spanned by $\mathrm{supp}(\mathbf{m})$; if the cardinality of the support is two, it contains a node in $I_1^{\mathrm{it}}$ by the last property and we declare the other node to be the center.
Let $\mathrm{supp}(\mathbf{m})=\{j_0,j_1,\dots,j_{\ell}\}$ and note that $\mathrm{supp}(\mathbf{m})\backslash \{j_0\}\subseteq I_1^{\mathrm{it}}$.

In this case, we can suppose that the word $w$ has the following property (again ignoring the sign of the corresponding Lie word)
\begin{itemize}
    \item the word ends with $j_0$, i.e. $i_r=j_0$. 
    
    \item All indices between two $j_0$'s in $w$ are distinct, since otherwise we have an index of $I_1^{\mathrm{it}}$ which appears at least twice between two $j_0$'s and thus $e(w)=0$ again from $[e_i,[e_i,x]]=0$ for all $i\in I_1^{\mathrm{it}}$. Similarly, all indices to the left of the left most $j_0$ are all distinct. 
\end{itemize}
In particular, $w$ must be of the form $w=w' \cdots w' w'$ ($k$ times) where $w'=j_{\ell}\cdots j_{1} j_0$. Hence by the super Jacobi identity and the fact that $[e_j,e(w')]=0$ for all $j\in\{j_1,\dots,j_{\ell}\}$ (again by $[e_i,[e_i,x]]=0$ for all $i\in I_1^{\mathrm{it}}$) we get
$$e(w)=[e(w'),[e(w'),\dots,[e(w'),[e(w'),e(w')]]\cdots]]=0$$
which is a contradiction and hence $\mathbf{m}\notin K$. The discussion so far shows that $\Delta^{\mathcal{G}}$ is contained in the right hand side of \eqref{rootdescp}. For the converse direction, set 
$$M:=\{\mathbf{m}\in K^1_2: m_i=2\  \forall i\in \mathrm{supp}(\mathbf{m})\}$$
and let $\mathbf{m}\in M$. From Theorem~\ref{multroo1} we obtain 
\begin{align*}\mathrm{mult}_{\mathcal{G}}(\mathbf{m})&=-{_{\bold{m}}\Pi^{\mathrm{mark}}_\mathcal{G}}(q)[q]+\frac{1}{2}(-1)^{|\mathbf{m}|/2-1}{_{\frac{\bold{m}}{2}}\Pi^{\mathrm{mark}}_\mathcal{G}}(q)[q]&\\&
=-{_{\bold{m}}\Pi^{\mathrm{mark}}_\mathcal{G}}(q)[q]+\frac{1}{2}(-1)^{|\mathbf{m}|/2-1}(-1)^{|\mathrm{supp}(\bold m)|-1}&\\&
=\frac{1}{2}+\frac{1}{2}(-1)^{|\mathbf{m}|/2+|\mathrm{supp}(\bold m)|}=1 \ \ \ \  \text{\  by \ Example~\ref{exchrhalb}(2)}   
\end{align*}
and thus $M\subseteq \Delta^{\mathcal{G}}$. It remains to show that each $\mathbf{m}\in P\backslash (K\cup M)$ is contained in $\Delta^{\mathcal{G}}$, which we will show by induction on the height of $\mathbf{m}$. The initial step of the induction is clear.

Let $\mathbf{m}\in P\backslash (K\cup M)$ and let $i_0\in I^{\mathrm{it}}_1$ be a node such that $m_{i_0}\geq m_i$ for all $i\in I_1^{\mathrm{it}}$. If $m_{i_0}\leq 1$, then $\mathbf{m}\in \Delta^{\mathcal{G}}$ by Lemma~\ref{someroots}. So suppose that $m_{i_0}\geq 2$ and set
$$J_{i_0}=\left\{j\in N_{\mathcal{G}}(i_0)\cap I^{\mathrm{it}}_1: m_j\in\{m_{i_0},m_{i_0}-1\},\ m_j\geq 2, \ N_{\mathcal{G}}(j)\cap I^{\mathrm{it}}_1\cap \mathrm{supp}(\mathbf{m})=\{i_0\}\right\}.$$
\textit{Case 1:} Assume in the first case that $J_{i_0}$ is non-empty. We set 
$$\mathbf{m}_{i_0}=\alpha_{i_0}+\sum_{j\in J_{i_0}}\alpha_j,\ \ \tilde{\mathbf{m}}=\mathbf{m}-\mathbf{m}_{i_0}.$$
We will show first $\tilde{\mathbf{m}}\in P$. Note that $\mathrm{supp}(\tilde{\mathbf{m}})=\mathrm{supp}(\mathbf{m})$ and hence the support of $\tilde{\mathbf{m}}$ is connected and at least of cardinality two. Given $t\in I_1^{\mathrm{it}}$ with $\tilde{m}_t\geq 2$ we consider two cases:

\textit{Case 1.1:} Suppose $i_0\in N_{\mathcal{G}}(t)$, then we have 
$\tilde{m}_{i_0}=m_{i_0}-1\geq \tilde{m}_{t}$ unless $t\notin \{i_0\}\cup J_{i_0}$ and $m_t=m_{i_0}$. In the latter case we have
$$N_{\mathcal{G}}(t)\cap I^{\mathrm{it}}_1\cap \mathrm{supp}(\mathbf{m})\neq \{i_0\}.$$
So let $j\in N_{\mathcal{G}}(t)\cap I^{\mathrm{it}}_1\cap \mathrm{supp}(\mathbf{m})$ with $j\neq i_0$. The assumption $j\in J_{i_0}$ would give
\begin{equation}\label{44iu} t\in N_{\mathcal{G}}(j)\cap I^{\mathrm{it}}_1\cap \mathrm{supp}(\mathbf{m})=\{i_0\}\implies t=i_0\end{equation}
which is impossible. Thus $j\notin \{i_0\}\cup J_{i_0}$ and $\tilde{m}_j=m_j>0$. So in both cases 
\begin{equation}\label{44t}\tilde{m}_t\leq \sum_{j\in N_{\mathcal{G}}(t)}\tilde{m}_j.\end{equation}

\textit{Case 1.2:} Suppose $i_0\notin N_{\mathcal{G}}(t)$. If $N_{\mathcal{G}}(t)\cap J_{i_0}=\emptyset$ we immediately get \eqref{44t} since $\mathbf{m}\in P$. So let $j'\in N_{\mathcal{G}}(t)\cap J_{i_0}$ which implies $t=i_0$ as in \eqref{44iu}. Again we obtain \eqref{44t} unless $m_{j'}=m_{i_0}-1\geq 2$. Since $\mathbf{m}\in P$, there must exist a node $j''\in N_{\mathcal{G}}(i_0)$ with $j''\neq j'$ and $m_{j''}>0$. Hence
$$\sum_{j\in N_{\mathcal{G}}(t)}\tilde{m}_j=\sum_{j\in N_{\mathcal{G}}(t),\ j\notin\{j',j''\}}\tilde{m}_j+(m_{i_0}-2)+\tilde{m}_{j''}\geq m_{i_0}-1=\tilde{m}_{i_0}.$$
So we finally obtain $\tilde{\mathbf{m}}\in P$. Assume now by contradiction $\tilde{\mathbf{m}}\in K\cup M$. By definition $\mathrm{supp}(\mathbf{m})$ spans a star graph and if the center would be different from $i_0$, it has to be an element of $J_{i_0}$ since $J_{i_0}$ is non-empty. However, by the definition of $J_{i_0}$ this forces $|\mathrm{supp}(\mathbf{m})|=2$. So either $i_0$ is the center or $|\mathrm{supp}(\mathbf{m})|=2$ where we can assume by convention that $i_0$ is the center of the star graph. We write 
$$\mathbf{m}=\tilde{\mathbf{m}}+\mathbf{m}_{i_0}=\sum_{i\in \mathrm{supp}(\mathbf{m})} k\alpha_i +\sum_{i\in \{i_0\}\cup J_{i_0}}\alpha_i=\sum_{i\in \{i_0\}\cup J_{i_0}}(k+1)\alpha_i+\sum_{i\in \mathrm{supp}(\mathbf{m})\backslash (\{i_0\}\cup J_{i_0})}k\alpha_i$$
for some $k\geq 2$. Since $\mathbf{m}\notin K\cup M$, we must have
$$\mathrm{supp}(\mathbf{m})\backslash (\{i_0\}\cup J_{i_0})\neq \emptyset.$$
If $j$ is a node in the above mentioned set we must have $j\in N_{\mathcal{G}}(i_0)\cap I_1^{\mathrm{it}}$ since $\mathrm{supp}(\mathbf{m})$ spans a star graph with center $i_0$ and 
$$m_j=k=m_{i_0}-1\geq 2\implies j\in J_{i_0}$$
which is a contradiction and thus $\tilde{\mathbf{m}}\notin K\cup M$. So by induction $\tilde{\mathbf{m}}\in \Delta^{\mathcal{G}}$ and $\mathbf{m}_{i_0}\in \Delta^{\mathcal{G}}$ by Lemma~\ref{someroots}(1) and thus Lemma~\ref{someroots}(2) gives $\mathbf{m}\in \Delta^{\mathcal{G}}$.


\textit{Case 2:} Assume in the second case that $J_{i_0}$ is empty and consider 
$$L_{i_0}=\{j\in N_{\mathcal{G}}(i_0): 2\leq m_j\leq m_{i_0}\}.$$

\textit{Case 2.1:} Assume first that $L_{i_0}$ is non-empty and choose $j_0\in L_{i_0}$ such that $m_{j_0}$ is maximal. We further set
$$S_{j_0}=\bigg\{k\in (N_{\mathcal{G}}(j_0)\cap I^{\mathrm{it}}_1)\backslash \{i_0\}: m_k=\sum_{p\in N_{\mathcal{G}}(k)}m_p\bigg\}$$
$$R_{j_0}=\left\{k\in (N_{\mathcal{G}}(j_0)\cap I^{\mathrm{it}}_1)\backslash \{i_0\}: \text{$ m_k+1=m_{j_0}=m_{i_0}\geq 3$ and $N_{\mathcal{G}}(k)\cap I^\mathrm{it}_1\cap \mathrm{supp}(\mathbf{m})\subseteq\{j_0\}$}\right\}$$

$$T_{j_0}=S_{j_0}\cup R_{j_0},\ \ \mathbf{m}_{i_0,j_0}=\alpha_{i_0}+\alpha_{j_0}+\sum_{t\in T_{j_0}}\alpha_t,\ \ \tilde{\mathbf{m}}=\mathbf{m}-\mathbf{m}_{i_0,j_0}.$$
We first show that $\tilde{\mathbf{m}}\in P$. Clearly $\mathrm{supp}(\tilde{\mathbf{m}})=\mathrm{supp}(\mathbf{m})$ and hence the support is connected and at least of cardinality two. First note that
\begin{equation}\label{444r}(N_{\mathcal{G}}(t')\setminus\{j_0\})\cap T_{j_0}=\emptyset\ \ \forall t'\in \{i_0\}\cup S_{j_0}.\end{equation}
To see this, let $t'\in S_{j_0}$ and $k\in (N_{\mathcal{G}}(t')\setminus\{j_0\})\cap S_{j_0}$. Then $$m_k=\sum_{p\in N_{\mathcal{G}}(k)}m_p\geq m_{t'}+m_{j_0}=\sum_{\ell\in N_{\mathcal{G}}(t')}m_\ell+m_{j_0}\geq m_k+m_{j_0}$$
which is impossible. If $k\in (N_{\mathcal{G}}(t')\setminus\{j_0\})\cap R_{j_0}$ we get again a contradiction:
$$t'\in N_{\mathcal{G}}(k)\cap I^\mathrm{it}_1\cap \mathrm{supp}(\mathbf{m})\subseteq \{j_0\}.$$
If $t'=i_0$, the proof of \eqref{444r} is similar and we omit the details. 
We fix $t\in I_1^{\mathrm{it}}$ with $\tilde{m}_t\geq 2$ and show next that \eqref{44t} holds.

$\bullet$ If $t=i_0$, we have 
$$\tilde{m}_{i_0}=m_{i_0}-1\leq m_{j_0}-1+\sum_{p\in N_{\mathcal{G}}(i_0)\setminus\{j_0\}}m_p=\sum_{p\in N_{\mathcal{G}}(i_0)}\tilde{m}_p.$$
where the last equation follows from \eqref{444r}.

$\bullet$ If $t=j_0$, we immediately get \eqref{44t} since $\tilde{m}_{j_0}\leq \tilde{m}_{i_0}$ and $i_0\in N_{\mathcal{G}}(j_0)$. 

$\bullet$ Next we assume $t\in N_{\mathcal{G}}(i_0)$. We can also assume $t\neq j_0$, since $t=j_0$ is already treated above. Then we must have $\tilde{m}_t=m_t$ and if $m_t<m_{i_0}$, we immediately get \eqref{44t}. Otherwise $m_t=m_{i_0}$ and since $J_{i_0}=\emptyset$ we must have an index $j\in N_{\mathcal{G}}(t)\cap I^{\mathrm{it}}_1\cap \mathrm{supp}(\mathbf{m})$ with $j\neq i_0$. Moreover, either $j\notin T_{j_0}\cup \{j_0\}$ and $\tilde{m}_j=m_j\geq 1$ or $j\in T_{j_0}\cup \{j_0\}$ and $\tilde{m}_j=m_j-1\geq 1$. This gives once more \eqref{44t}.

$\bullet$ Next we assume $t\in S_{j_0}$. We get
$$\tilde{m}_t=m_t-1=\sum_{j\in N_{\mathcal{G}}(t)}m_j-1=\sum_{j\in N_{\mathcal{G}}(t)}\tilde{m}_j$$
where the last equation is again implied by \eqref{444r}.

$\bullet$ Next we consider $t\in R_{j_0}$. Then $\tilde{m}_t=m_t-1=m_{j_0}-2<m_{j_0}-1=\tilde{m}_{j_0}$ and $j_0\in N_{\mathcal{G}}(t)$. Thus \eqref{44t} is immediate.

$\bullet$ Next we consider $t\in N_{\mathcal{G}}(j_0)$. We can also assume that $t\notin T_{j_0}\cup\{i_0\}$ and $t \notin N_{\mathcal{G}}(i_0)$ since these cases are treated above. If there exists a node $r\in N_{\mathcal{G}}(t)\setminus\{j_0\}$ with $r\in S_{j_0}$ we get
$$\tilde{m}_r=m_r-1=m_{j_0}-1+m_t+\sum_{p\in N_{\mathcal{G}}(r)\setminus\{t,j_0\}}m_p\implies \tilde{m}_r\geq \tilde{m}_t.$$
Thus \eqref{44t} is obtained and we are done.
Hence we can suppose that $N_{\mathcal{G}}(t)\setminus\{j_0\}\cap T_{j_0}=\emptyset$ (note that $N_{\mathcal{G}}(t)\setminus\{j_0\}\cap R_{j_0}=\emptyset$ is automatic) and therefore

$$\tilde{m}_t=m_t<m_{j_0}+\sum_{k\in N_{\mathcal{G}}(t)\setminus\{j_0\}}m_k=\sum_{j\in N_{\mathcal{G}}(t)}\tilde{m}_j+1$$

$\bullet$ Next we consider $t\in N_{\mathcal{G}}(r)$  for some $r\in T_{j_0}$. Again we can assume $t\notin \{i_0,j_0\}\cup T_{j_0}$ since these cases are treated above. Then we have $r\in S_{j_0}$ and
$$m_r=m_{j_0}+m_t+\sum_{s\in N_{\mathcal{G}}(r)\setminus\{t,j_0\}}m_s\implies m_t<m_r.$$
Hence
$$\tilde{m}_t=m_t\leq m_r-1=\tilde{m}_r$$
and \eqref{44t} is obtained.

$\bullet$ Next we consider $t\notin \{i_0,j_0\}\cup N_{\mathcal{G}}(i_0)\cup N_{\mathcal{G}}(j_0)\bigcup_{r\in T_{j_0}}N_{\mathcal{G}}(r)$. Then
$$\tilde{m}_t=m_t\leq  \sum_{j\in N_{\mathcal{G}}(t)}m_j=\sum_{j\in N_{\mathcal{G}}(t)}\tilde{m}_j$$
which finally shows $\tilde{\mathbf{m}}\in P$. The proof that $\tilde{\mathbf{m}}\notin K\cup M$ is similar to Case 1 and we omit the details. Hence, by induction, $\tilde{\mathbf{m}}\in\Delta^{\mathcal{G}}$ and $\mathbf{m}_{i_0,j_0}\in \Delta^{\mathcal{G}}$ by Lemma~\ref{someroots}(1) as $T_{j_0}\cup\{i_0\}\cup\{j_0\}$ is connected. Thus by Lemma~\ref{someroots}(2) again $\mathbf{m}\in\Delta^{\mathcal{G}}$.


\textit{Case 2.2:} Now we consider the case $L_{i_0}=\emptyset$, i.e., for any
$j\in N_{\mathcal{G}}(i_0)\cap \mathrm{supp}(\mathbf{m})$ we have $m_j=1$. We fix $j_0\in N_{\mathcal{G}}(i_0)\cap \mathrm{supp}(\mathbf{m})$ and define as before
$$S_{j_0}=\bigg\{k\in (N_{\mathcal{G}}(j_0)\cap I^{\mathrm{it}}_1)\backslash \{i_0\}: m_k=\sum_{p\in N_{\mathcal{G}}(k)}m_p\bigg\}.$$

\textit{Case 2.2.1:}
Suppose that the full subgraph of $\mathcal{G}$ spanned by $\mathrm{supp}(\mathbf{m})\setminus \{j_0\}$ is connected. Define
$$\mathbf{m}_{i_0,j_0}=\alpha_{i_0}+\alpha_{j_0}+\sum_{t\in S_{j_0}}\alpha_t\in \Delta^\mathcal{G},\ \ \tilde{\mathbf{m}}=\mathbf{m}-\mathbf{m}_{i_0,j_0}.$$
As before, induction and Lemma~\ref{someroots} will complete the proof in this case, once we have shown $\tilde{\mathbf{m}}\in P\backslash (K\cup M)$.
Note that $m_{j_0}=1$ and $m_k\ge 2$ for all $k\in S_{j_0}$
implies that $\mathrm{supp}(\tilde{\mathbf{m}}) = \mathrm{supp}(\mathbf{m})\setminus \{j_0\}.$ Moreover, $|\mathrm{supp}(\tilde{\mathbf{m}})|\geq 2$ since otherwise
\begin{equation}\label{assasin}2\leq m_{i_0}\leq \sum_{k\in N_{\mathcal{G}}(i_0)}m_k=1.\end{equation}
So let $t\in I_1^{\mathrm{it}}$ with $\tilde{m}_t\ge 2$ and assume additionally that $t\in N_{\mathcal{G}}(r)$ for some $r\in \{i_0, j_0\}\cup S_{j_0}$; otherwise \eqref{44t} is immediate from $\mathbf{m}\in P$. Also note that $t\in N_{\mathcal{G}}(i_0)$ is not possible since on the one hand $m_t\geq \tilde{m}_t\geq 2$ and on the other hand $m_t=1$.

$\bullet$ For $t =i_0$, we have $\tilde{m}_p = m_p$ for all $p\in N_{\mathcal{G}}(i_0)\setminus\{j_0\}$ since each element of $N_{\mathcal{G}}(i_0)\setminus\{j_0\}$ is not contained in $S_{j_0}.$ Hence \eqref{44t} is obtained.

$\bullet$ Consider the case $t\in S_{j_0}$ and assume additionally $t\neq i_0$. Note that in this case if $p\in N_{\mathcal{G}}(t)\setminus\{j_0\}$, then $p\notin S_{j_0}$ and also $p\neq i_0$. Thus $m_p=\tilde{m}_p$ and \ref{44t} follows. 

$\bullet$ Consider the case $t\in N_{\mathcal{G}}(j_0)$ and assume additionally $t\notin \{i_0\}\cup S_{j_0}$ by the cases treated above. 
If $p\in N_{\mathcal{G}}(t)\setminus\{j_0\}$ and $p\in S_{j_0}$, then we have $t, j_0\in N_{\mathcal{G}}(p)$ and
$m_p = \sum_{s\in N_{\mathcal{G}}(p)}m_s \ge 1+ m_t$. This implies that $\tilde{m_p} \ge m_t$ and \ref{44t} holds. So suppose $p\notin S_{j_0}$ for all $p\in N_{\mathcal{G}}(t)\setminus\{j_0\}$, then we have $\tilde{m}_p = m_p$ for all $p\in N_{\mathcal{G}}(t)\setminus\{j_0\}$. Thus,  
$$\tilde{m}_t=m_t\leq \sum_{k\in N_{\mathcal{G}}(t)\setminus\{j_0\}}m_k +m_{j_0}-1=\sum_{k\in N_{\mathcal{G}}(t)}\tilde{m}_k$$
where the inequality follows from $\mathbf{m}\in P$ and $t\notin S_{j_0}$. Therefore \ref{44t} follows.

$\bullet$ Consider the case $t\in N_{\mathcal{G}}(r)$ for some $r\in S_{j_0}.$ By the cases treated above we can assume $t\notin \{i_0\}\cup S_{j_0}$. Thus $\tilde{m}_t =m_t\le m_r-1 = \tilde{m_r}$ as $t\neq j_0$ and $r\in S_{j_0}$. This implies \ref{44t}.

Hence $\tilde{\mathbf{m}}\in P$ and since $\tilde{m}_j = 1$ for some $j\in \mathrm{supp}(\tilde{\mathbf{m}})$, we also have $\tilde{\mathbf{m}}\notin K\cup M$.

\textit{Case 2.2.2:}
Suppose that the full subgraph of $\mathcal{G}$ spanned by $\mathrm{supp}(\mathbf{m})\setminus \{j_0\}$ is not connected and denote the connected components by $\mathcal{G}_{(1)},\dots,\mathcal{G}_{(r)}$, $r\geq 2$. Let $V(\mathcal{G}_{(i)})$ be the vertex set of $\mathcal{G}_{(i)}$  and assume that $i_0\in V(\mathcal{G}_{(1)})$. Set $V_{i_0,j_0}=\bigcup_{s=2}^r V(\mathcal{G}_{(s)})$ and

$$\mathbf{m}_{i_0,j_0}=\alpha_{i_0}+\alpha_{j_0}+\sum_{k\in V_{i_0,j_0}}m_k\alpha_k+\sum_{k\in S_{j_0}\cap V(\mathcal{G}_1)}\alpha_k=(n_s:s\in I),\ \ \tilde{\mathbf{m}}=\mathbf{m}-\mathbf{m}_{i_0,j_0}$$ 

Note that $\mathrm{supp}(\mathbf{m}_{i_0,j_0})$ is connected and $|\mathrm{supp}(\mathbf{m}_{i_0,j_0})|\geq 2$. 
If $t\in I_1^{\mathrm{it}}$ with $n_t\geq 2$, we must have $\{t\}\cup N_\mathcal{G}(t)\subseteq V_{i_0,j_0}\cup\{j_0\}$ and thus
 $$n_t=m_t\leq\sum_{p\in N_\mathcal{G}(t)}m_p=\sum_{p\in N_\mathcal{G}(t)}n_p\implies \mathbf{m}_{i_0,j_0}\in P.$$ Since $n_{i_0}=1$ we also have $\mathbf{m}_{i_0,j_0}\in P\backslash (K\cup M)$ and induction gives $\mathbf{m}_{i_0,j_0}\in \Delta^\mathcal{G}$.

Next we will show that $\tilde{\mathbf{m}}\in P$. The full subgraph generated by $\mathrm{supp}(\tilde{\mathbf{m}})$ is $\mathcal{G}_{(1)}$ and hence connected. Also  $|\mathrm{supp}(\tilde{\mathbf{m}})|\geq 2$ because $m_{i_0}\geq 2,$  and  $N_{\mathcal{G}}(i_0)\cap \mathrm{supp}({\mathbf{m}})$ must contain a node different from $j_0$ (otherwise we get a coontradiction as in \eqref{assasin}). Fix $t\in I^\mathrm{it}_1$ with $\tilde{m}_t\geq 2$ and we claim that \eqref{44t} holds.

$\bullet$ If $ t=i_0$, we have $n_p=0$ for all $p\in (N_\mathcal{G}(i_0)\cap \mathrm{supp}(\mathbf{m}))\setminus\{j_0\}$ and thus
$$\tilde{{m}}_{i_0}=m_{i_0}-1\leq \sum_{p\in N_\mathcal{G}(i_0)\setminus\{j_0\}}m_p= \sum_{p\in N_\mathcal{G}(i_0)\setminus\{j_0\}} \tilde{m}_{p}.$$

$\bullet$ If $t\in S_{j_0}\cap V(\mathcal{G}_1)$, and $p\in N_\mathcal{G}(t)\setminus\{j_0\}$, then $p\notin S_{j_0}$ and also $p\neq i_0$ (otherwise $2\leq m_t=1$). Thus $\tilde{m}_p=m_p$ and we get 
$$ \tilde{m}_t=m_t-1=\sum_{p\in N_\mathcal{G}(t)\setminus\{j_0\}}m_p=\sum_{p\in N_\mathcal{G}(t)}\tilde{m}_p.$$

$\bullet$ If $t\in N_\mathcal{G}(j_0)$ we can also assume $t\notin \{i_0\}\cup (S_{j_0}\cap  V(\mathcal{G}_{(1)}))$ by the cases treated above. If $t\in V_{i_0,j_0} $  we get $\tilde{m}_t= m_t-n_t=0$ which is absurd. Otherwise $t\in V(\mathcal{G}_{(1)})$ and $\tilde{m}_t= m_t$. If there exists $p\in N_{\mathcal{G}}(t)\setminus\{j_0\}$ with $p\in S_{j_0}\cap V(\mathcal{G}_{(1)})$ we get 
$$t, j_0\in N_{\mathcal{G}}(p)\implies m_p = \sum_{s\in N_{\mathcal{G}}(p)}m_s \ge 1+ m_t\implies \tilde{m}_p \geq \tilde{m}_t\implies \eqref{44t}.$$ So suppose $p\notin S_{j_0}\cap V(\mathcal{G}_{(1)})$ for all $p\in N_{\mathcal{G}}(t)\setminus\{j_0\}$, then  $p\neq i_0$ because $m_k=1\ \forall k\in N_\mathcal{G}(i_0)\cap \mathrm{supp}(\mathbf{m})$, and $t\in  V(\mathcal{G}_{(1)})$ implies $ p\notin V_{i_0,j_0}$.  Hence $\tilde{m}_p = m_p$  for all $p\in N_{\mathcal{G}}(t)\setminus\{j_0\}$ and \ref{44t} follows.

$\bullet$ If $t\in N_\mathcal{G}(r)$ for some $r\in S_{j_0}\cap  V(\mathcal{G}_{(1)})$, we will assume additionally $t\notin \{i_0\}\cup (S_{j_0}\cap  V(\mathcal{G}_{(1)}))$ by the cases treated above and note that $t=j_0$ is impossible. Thus
$$m_r=m_{j_0}+\sum_{p\in N_\mathcal{G}(r)\setminus\{j_0\}}m_p\implies m_t<m_r\implies \tilde{m}_t=m_t\leq m_r-1=\tilde{m}_r\implies \eqref{44t}.$$

$\bullet$ If $t\notin \left(N_\mathcal{G}(j_0)\bigcup_{r\in (S_{j_0}\cap V(\mathcal{G}_{(1)}))}N_\mathcal{G}(r) \right)$, then we have (note that $t\in V_{i_0,j_0}$ is impossible) 
$$\tilde{m}_t\leq m_t\leq  \sum_{j\in N_{\mathcal{G}}(t)}m_j=\sum_{j\in N_{\mathcal{G}}(t)}\tilde{m}_j.$$

Thus we have $\tilde{\mathbf{m}}\in P$. Since $\tilde{m}_k=1$ for some $k\in \mathrm{supp}({\tilde{\mathbf{m}}})$, we have $\tilde{\mathbf{m}}\notin (K\cup M)$ and hence by induction $\tilde{\mathbf{m}}\in \Delta^\mathcal{G}$. Again Lemma~\ref{someroots}(2) finishes the proof. 
\end{proof}
 

\section{Right-angled Coxeter Groups and Partially commutative superalgebras}\label{racgsection}

In this section, we explore the connection between right-angled Coxeter groups (RACG) and partially commutative Lie superalgebras. We focus on the case when $\mathcal{G}$ is a marked graph with $I = I_1^{\mathrm{it}}$, i.e., all the vertices are odd isotropic. In this special case, we prove that there is a natural $\mathbb{C}$-vector isomorphism  between the universal enveloping algebra of $\mathcal{P}(\mathcal G)$ and  the group algebra of the right-angled Coxeter group associated with $\mathcal{G}$.
This will allow to express the Hilbert series of RACG in terms of marked multivariate independece series.
\subsection{}
When $I = I_1^{\mathrm{it}}$, the universal enveloping algebra $\mathbf{U}_{\mathcal{G}}$  of $\mathcal{P}(\mathcal G)$ is a $\mathbb C$-associative superalgebra generated by
$e_i, i\in I$ with relations $$e_i^2 = 0,\ i\in I, \, \, e_ie_j =- e_je_i,\ (i, j)\notin E.$$
By Proposition~\ref{denomi}, the Hilbert series of $\mathbf{U}_{\mathcal{G}}$ is given by 
$$\mathcal{H}_{\mathbf{U}_{\mathcal{G}}}(\bold x) = \frac{1}{I(\mathcal{G}, -\mathbf{x})},$$
where in the case of $I= I_1^{\mathrm{it}}$ we have 
$$ I(\mathcal{G}, \bold x) = \sum\limits_{S\in \widetilde{\mathcal{I}}(\mathcal{G})} \sum\limits_{\ell_i\ge 1, i\in S} \prod_{i\in S} x_i^{\ell_i} = \sum\limits_{S\in \widetilde{\mathcal{I}}(G)}\left(\prod_{i\in S}\frac{x_i}{1-x_i}\right)$$
and $\widetilde{\mathcal{I}}(G)$ is the set of finite independent subsets of $\mathcal{G}$. Thus we have
$$\mathcal{H}_{\mathbf{U}_{\mathcal{G}}}(\bold x) = \left(\sum\limits_{S\in \widetilde{\mathcal{I}}(G)}\left(\prod\limits_{i\in S}\frac{-x_i}{1+x_i}\right)\right)^{-1}.$$

\subsection{}
Recall that the right-angled Coxeter group associated with $\mathcal{G}$ is defined to be the group generated by the generators $s_i, i\in I$ and with relations $$s_i^2 = 1,\ i\in I, \, \, s_is_j = s_js_i,\ (i, j)\notin E.$$
The Poincare series of $W_{\mathcal{G}}$ defined by $$P_{\mathcal{G}}(t) = \sum\limits_{w\in W_{\mathcal{G}}} t^{\ell(w)}$$
has a well-known description which we summarize below; here, $\ell(w)$ is the spherical growth function (or length function with respect to the Coxeter generators $s_i, i\in I$). The following can be found in \cite[Prop. 17.4.2]{Davis2015}. 
\begin{prop}\label{davis} Let $W_{\mathcal{G}}$ be the right-angled Coxeter group associated with $\mathcal{G}$. Then we have
$$P_{\mathcal{G}}(t) = \frac{1}{I_{\mathcal{G}}\left(\frac{-t}{1+t}\right)}$$
where $I_{\mathcal{G}}(t) = \sum\limits_{k\ge 0} a_k t^k$ denotes the usual one-variable independence polynomial of $G$, i.e. $a_k$ counts the number of independent subsets of $\mathcal{G}$ of size $k.$ \qed
\end{prop}
Using partially commutative Lie superalgebras, we will refine the above result.
\subsection{} The group algebra of $W_{\mathcal G}$, denoted as $\mathbb{C}[W_{\mathcal{G}}]$, is a $\mathbb C$-associative algebra generated by the generators $e(s_i), 1\le i\le n$ with relations
$$e(s_i)^2 = 1,\ i\in I\ \ e(s_i)e(s_j) = e(s_j)e(s_i),\ (i, j)\notin E.$$ For an element $w= s_{i_1}\cdots s_{i_k}\in W_{\mathcal{G}}$, we set $e(w) = e(s_{i_1})\cdots e(s_{i_k})$. Note that $e(w)$ is independent of the expression of $w$ and  $\{e(w) : w\in W_{\mathcal{G}}\}$ forms a basis for $\mathbb{C}[W_G]$. Moreover, if $w=s_{i_1}\cdots s_{i_k}$ is a reduced expression of $w$ we set $\bold x(w) = x_{i_1}\cdots x_{i_k}$. Note that $\bold x(w)$ depends only on $w$, not on the reduced expression chosen. The Hilbert series of $\mathbb{C}[W_{\mathcal{G}}]$ is given by
$$\mathcal{H}_{\mathcal{G}}(\bold x) = \sum\limits_{w\in W_{\mathcal{G}}} \bold x(w).$$

\begin{thm}\label{racghilber}
We have a $\mathbb C$-vector space isomorphism between $\mathbf{U}_{\mathcal{G}}$ and $\mathbb{C}[W_{\mathcal{G}}]$ which maps a basis element $e_{i_1}\cdots e_{i_k}$ of $\mathbf{U}_{\mathcal{G}}$ to the expression $e(s_{i_1})\cdots e(s_{i_k}).$ In particular, it is a grade-preserving isomorphism, and we have
$$\mathcal{H}_{\mathcal{G}}(\bold x) = \frac{1}{I({\mathcal{G}},-\bold x)}=\left(\sum\limits_{S\in \widetilde{\mathcal{I}}(G)}\left(\prod_{i\in S}\frac{-x_i}{1+x_i}\right)\right)^{-1}.$$ \qed
\end{thm}
Hence, we recover Proposition~\ref{davis} by substituting $x_i=t$ for all $i\in I$.
With the notations of Section \ref{peographsection} we obtain the Hilbert series of RACG for PEO-graphs.
\begin{cor}
 Let $\mathcal{G}$ be a PEO-graph with a countable vertex set $I$. Then we have 
 $$\mathcal{H}_{\mathcal{G}}(\bold x) =\sum\limits_{\boldsymbol{m}\in \mathbb{Z}^I_+}\left(\sum\limits_{\boldsymbol{\lambda}\in S(\bold m)}(-1)^{|\mathbf{m}|+\sum_{j\in \mathrm{supp}(\bold{m})}\ell(\boldsymbol{\lambda}_{j})}\prod _{j\in \mathrm{supp}(\bold{m})} \binom{\sum\limits_{i\in \mathcal{G}_j}\ell(\boldsymbol{\lambda}_{i})}{\ell(\boldsymbol{\lambda}_{j})} \frac{\ell(\boldsymbol{\lambda}_{j})!}{\prod_{k=1}^{\infty}(d^{\boldsymbol{\lambda}_{j}}_{k}!)}\right) x^\mathbf{m}$$ \qed
 
\end{cor}

\bibliographystyle{plain}
\bibliography{chromatic}

\begin{thebibliography}{10}

\bibitem{AKV18}
G.~Arunkumar, Deniz Kus, and R.~Venkatesh.
\newblock Root multiplicities for {B}orcherds algebras and graph coloring.
\newblock {\em J. Algebra}, 499:538--569, 2018.

\bibitem{BMP92}
Yuri~A. Bahturin, Alexander~A. Mikhalev, Viktor~M. Petrogradsky, and Mikhail~V.
  Zaicev.
\newblock {\em Infinite-dimensional {L}ie superalgebras}, volume~7 of {\em De
  Gruyter Expositions in Mathematics}.
\newblock Walter de Gruyter \& Co., Berlin, 1992.

\bibitem{BJN19}
Riccardo Biagioli, Fr\'ed\'eric Jouhet, and Philippe Nadeau.
\newblock 321-avoiding affine permutations and their many heaps.
\newblock {\em J. Combin. Theory Ser. A}, 162:271--305, 2019.

\bibitem{B93}
Norman Biggs.
\newblock {\em Algebraic graph theory}.
\newblock Cambridge Mathematical Library. Cambridge University Press,
  Cambridge, second edition, 1993.

\bibitem{Borcherds1992}
Richard~E. Borcherds.
\newblock Monstrous moonshine and monstrous {L}ie superalgebras.
\newblock {\em Invent. Math.}, 109(2):405--444, 1992.

\bibitem{Davis2015}
Michael~W. Davis.
\newblock The geometry and topology of {C}oxeter groups.
\newblock In {\em Introduction to modern mathematics}, volume~33 of {\em Adv.
  Lect. Math. (ALM)}, pages 129--142. Int. Press, Somerville, MA, 2015.

\bibitem{Huang2020}
Yi-Zhi Huang.
\newblock Generators, spanning sets and existence of twisted modules for a
  grading-restricted vertex (super)algebra.
\newblock {\em Selecta Math. (N.S.)}, 26(4):Paper No. 62, 42, 2020.

\bibitem{Jurisich2004}
Elizabeth Jurisich and Robert Wilson.
\newblock A generalization of {L}azard's elimination theorem.
\newblock {\em Comm. Algebra}, 32(10):4037--4041, 2004.

\bibitem{Kang1998}
Seok-Jin Kang.
\newblock Graded {L}ie superalgebras and the superdimension formula.
\newblock {\em J. Algebra}, 204(2):597--655, 1998.

\bibitem{LY54}
R.~C. Lyndon.
\newblock On {B}urnside's problem.
\newblock {\em Trans. Amer. Math. Soc.}, 77:202--215, 1954.

\bibitem{Moriwaki2023}
Yuto Moriwaki.
\newblock Two-dimensional conformal field theory, full vertex algebra and
  current-current deformation.
\newblock {\em Adv. Math.}, 427:Paper No. 109125, 74, 2023.

\bibitem{Mu12}
Ian~M. Musson.
\newblock {\em Lie superalgebras and enveloping algebras}, volume 131 of {\em
  Graduate Studies in Mathematics}.
\newblock American Mathematical Society, Providence, RI, 2012.

\bibitem{SA2024}
Shushma Rani and G.~Arunkumar.
\newblock A study on free roots of {B}orcherds-{K}ac-{M}oody {L}ie
  superalgebras.
\newblock {\em J. Combin. Theory Ser. A}, 204:Paper No. 105862, 48, 2024.

\bibitem{Ray06}
Urmie Ray.
\newblock {\em Automorphic forms and {L}ie superalgebras}, volume~5 of {\em
  Algebra and Applications}.
\newblock Springer, Dordrecht, 2006.

\bibitem{Scheithauerbook}
Nils~R. Scheithauer.
\newblock Lie algebras, vertex algebras, and automorphic forms.
\newblock In {\em Developments and trends in infinite-dimensional {L}ie
  theory}, volume 288 of {\em Progr. Math.}, pages 151--168. Birkh\"auser
  Boston, Boston, MA, 2011.

\bibitem{SH62}
A.~I. {S}hirshov.
\newblock On the bases of a free {L}ie algebra.
\newblock {\em Algebra i Logika Sem.}, 1(1):14--19, 1962.

\bibitem{ML24}
Michael Vaughan-Lee.
\newblock {\em Bases for free {L}ie superalgebras}.
\newblock arXiv:2401.01174.

\bibitem{VV2015}
R.~Venkatesh and Sankaran Viswanath.
\newblock Chromatic polynomials of graphs from {K}ac-{M}oody algebras.
\newblock {\em J. Algebraic Combin.}, 41(4):1133--1142, 2015.

\bibitem{Vi89}
G\'erard~Xavier Viennot.
\newblock Heaps of pieces. {I}. {B}asic definitions and combinatorial lemmas.
\newblock In {\em Graph theory and its applications: {E}ast and {W}est
  ({J}inan, 1986)}, volume 576 of {\em Ann. New York Acad. Sci.}, pages
  542--570. New York Acad. Sci., New York, 1989.

\bibitem{Wa01}
Minoru Wakimoto.
\newblock {\em Infinite-dimensional {L}ie algebras}, volume 195 of {\em
  Translations of Mathematical Monographs}.
\newblock American Mathematical Society, Providence, RI, 2001.
\newblock Translated from the 1999 Japanese original by Kenji Iohara, Iwanami
  Series in Modern Mathematics.

\bibitem{Xu98}
Xiaoping Xu.
\newblock {\em Introduction to vertex operator superalgebras and their
  modules}, volume 456 of {\em Mathematics and its Applications}.
\newblock Kluwer Academic Publishers, Dordrecht, 1998.

\end{thebibliography}

\end{document}